\documentclass[11pt]{article}

\usepackage{amsmath}
\usepackage{amsfonts}
\usepackage{amssymb}
\usepackage{amsthm}
\usepackage{breqn}
\usepackage{setspace}
\usepackage{fullpage}
\usepackage{enumitem}
\usepackage{bbold} 
\usepackage{comment}
\usepackage{hyperref}
\usepackage{bbm}
\usepackage{tikz}
\usepackage{comment}
\usepackage{enumitem}
\usepackage{mathtools} 
\usetikzlibrary{patterns,arrows,decorations.pathreplacing}
\newtheorem{lemma}{Lemma}
\newtheorem{theorem}{Theorem}

\newtheorem{claim}{Claim}

\newtheorem{remark}{Remark}

\newcommand{\F}{\mathcal{F}}

\newcommand{\N}{\mathcal{N}}

\DeclareMathOperator{\ex}{ex}

\begin{document}
\title{The Maximum Number of Paths of Length Three in a Planar Graph}

\author{
Andrzej Grzesik\thanks{Jagiellonian University, Krak\'ow, Poland. email: \texttt{Andrzej.Grzesik@uj.edu.pl}} \and
Ervin Gy\H{o}ri\thanks{Alfr\'ed R\'enyi Institute of Mathematics. email: \texttt{gyori.ervin@renyi.mta.hu}} \footnotemark[4] \and
Addisu Paulos\thanks{Addis Ababa University, Addis Ababa. email: \texttt{addisu\_2004@yahoo.com}} \thanks{Central European University, Budapest.} \and
Nika Salia\thanks{Alfr\'ed R\'enyi Institute of Mathematics. email: \texttt{salia.nika@renyi.hu}}\hspace{.08em}
\footnotemark[4] \and 
Casey Tompkins\thanks{Discrete Mathematics Group, Institute for Basic Science (IBS), Daejeon, Republic of Korea.} \thanks{Karlsruhe Institute of Technology, Karlsruhe, Germany. email:\texttt{ctompkins496@gmail.com}} \and 
Oscar Zamora\thanks{Centro de Investigaci\'on en Matem\'atica Pura y Aplicada (CIMPA), Escuela de Matem\'atica,
Universidad de Costa Rica. San Jos\'e. email: \texttt{oscar.zamoraluna@ucr.ac.cr}}
\;\footnotemark[4]
}
\maketitle
\begin{abstract}
Let $f(n,H)$ denote the maximum number of copies of $H$ possible in an $n$-vertex planar graph. The function $f(n,H)$ has been determined when $H$ is a cycle of length $3$ or $4$ by Hakimi and Schmeichel and when $H$ is a complete bipartite graph with smaller part of size 1 or 2 by Alon and Caro. We determine $f(n,H)$ exactly in the case when $H$ is a path of length 3.
\end{abstract}
\section{Introduction and main result}
In recent times, generalized versions of the extremal function $\ex(n,H)$ have received considerable attention.  For graphs $G$ and $H$, let $\N(H,G)$ denote the number of subgraphs of $G$ isomorphic to~$H$ (referred to as \emph{copies} of $H$). Let $\F$ be a family of graphs, then a graph $G$ is said to be $\F$-free if it contains no graph from $\F$ as a subgraph. Alon and Shikhelman~\cite{ALS2016} introduced the following generalized extremal function (stated in higher generality in~\cite{add}),
\[
\ex(n,H,\F)=\max \{\N(H,G): \mbox{$G$ is an $\F$-free graph on $n$ vertices} \}. 
\]
If $\F=\{F\}$, we simply write $\ex(n,H,F)$. The earliest result of this type is due to Zykov~\cite{zykov} (and also independently by Erd\H{o}s~\cite{E1962}), who determined $\ex(n,K_s,K_t)$ exactly for all $s$ and $t$. Erd\H{o}s conjectured that asymptotically $\ex(n,C_5,C_3) = (\frac{n}{5})^5$ (where the lower bound comes from considering a blown up $C_5$).  This conjecture was finally verified quarter of a century later by Hatami, Hladk\'y, Kr\'al, Norine and Razborov~\cite{HHKNR2013} and independently by Grzesik~\cite{G2012}. Recently, the asymptotic value of $\ex(n,C_k,C_{k-2})$ was determined for every odd $k$ by  Grzesik and Kielak~\cite{GK2018}. In the opposite direction, the extremal function $\ex(n,C_3,C_5)$ was considered by Bollob\'as and Gy\H{o}ri~\cite{BGy2008}. Their results were subsequently improved in the papers~\cite{ALS2016},~\cite{C5C3} and~\cite{C5C3v2}, but the problem of determining the correct asymptotic remains open. The problem of maximizing $P_\ell$ copies in a $P_k$-free graph was investigated in~\cite{plpk}.

It is interesting that although maximizing copies of a graph $H$ in the class of $F$-free graphs has been investigated heavily, maximizing $H$-copies in other natural graph classes has received less attention.  In the setting of planar graphs such a study was initiated by Hakimi and Schmeichel~\cite{hakimi}. Let $f(n,H)$ denote the maximum number of copies of $H$ possible in an $n$-vertex planar graph. Observe that $f(n,H)$ is equal to $\ex(n,H,\F)$ where $\F$ is the family of $K_{3,3}$ or $K_5$ subdivisions~\cite{karatowski}.  
The case when $H$ is a clique and $\F$ is a family of clique minors has also been investigated (see, for example,~\cite{fox},~\cite{k5},~\cite{kcliques}).

Hakimi and Schmeichel determined the function $f(n,H)$ when $H$ is a triangle or cycle of length four.  Moreover, they classified the extremal graphs attaining this bound (a small correction to their result was given in~\cite{ahmad}).
\begin{theorem}[Hakimi and Schmeichel~\cite{hakimi}]\label{hakimi}
Let $G$ be a maximal planar graph with $n\geq 6$ vertices, then $\N(C_3,G)\leq 3n-8$. 
This bound is attained if and only if $G$ is a graph is obtained from $K_3$ by recursively placing a vertex inside a face and joining the vertex to the three vertices of that face (graphs constructed in this way are referred to as Apollonian networks).
\end{theorem}
\begin{theorem}[Hakimi and Schmeichel~\cite{hakimi}, Alameddine~\cite{ahmad}] \label{ahmad_hakimi}
Let $G$ be a maximal planar graph with $n\geq 5$ vertices, then $\N(C_4,G)\leq \frac{1}{2}(n^2+3n-22)$. 
For $n\neq 7,8$, the bound is attained if and only if $G$ is the graph shown in Figure~\ref{qq}(A). For $n=7$, the bound is attained if and only if $G$ is the graph in Figure~\ref{qq}($A$)~or~(B). For $n=8$, the bound is attained if and only if $G$ is the graph in Figure~\ref{qq}($A$)~or~($C$).
\end{theorem}
\begin{figure}[h]
\centering
\begin{tikzpicture}[scale=0.5]
\foreach \x in{1,2,3,4,5}{\draw[fill=black](0,\x)circle(4pt);};
\draw[fill=black](-3,0)circle(5pt);
\draw[fill=black](3,0)circle(5pt);
\foreach \x in{1,2,4}{\draw[thick](0,\x)--(0,\x+1);}
\draw (0,3.625)node{$\vdots$};
\foreach \x in{1,2,3,4,5}{\draw[thick](-3,0)--(0,\x)--(3,0);}
\draw[black,thick](-3,0)--(3,0);
\node at (0,-1) {$(A)$ $F_n$};
\end{tikzpicture}\qquad
\begin{tikzpicture}[scale=0.5]
\foreach \x in{1,2,4,5}{\draw[fill=black](0,\x)circle(4pt);};
\draw[fill=black](-3,0)circle(4pt);
\draw[fill=black](-0.6,2.4)circle(4pt);
\draw[fill=black](3,0)circle(4pt);
\foreach \x in{1,4}{\draw[thick](0,\x)--(0,\x+1);}
\foreach \x in{1,2,4,5}{\draw[thick](-3,0)--(0,\x)--(3,0);}
\draw[black,thick](-3,0)--(3,0)(0,2)--(0,4)(-3,0)--(-0.6,2.4)--(0,4)(-0.6,2.4)--(0,2);
\node at (0,-1) {$(B)$};
\end{tikzpicture}\qquad
\begin{tikzpicture}[scale=0.5]
\foreach \x in{1,2,4,5}{\draw[fill=black](0,\x)circle(4pt);};
\draw[fill=black](-3,0)circle(4pt);
\draw[fill=black](-0.6,2.4)circle(4pt);
\draw[fill=black](0.6,2.4)circle(4pt);
\draw[fill=black](3,0)circle(4pt);
\foreach \x in{1,4}{\draw[thick](0,\x)--(0,\x+1);}
\foreach \x in{1,2,4,5}{\draw[thick](-3,0)--(0,\x)--(3,0);}
\draw[black,thick](-3,0)--(3,0)(0,2)--(0,4)(-3,0)--(-0.6,2.4)--(0,4)(-0.6,2.4)--(0,2)(3,0)--(0.6,2.4)--(0,4)(0.6,2.4)--(0,2);
\node at (0,-1) {$(C)$};
\end{tikzpicture}
\caption{Planar graphs maximizing the number of cycles of length $4$.}
\label{qq}
\end{figure}
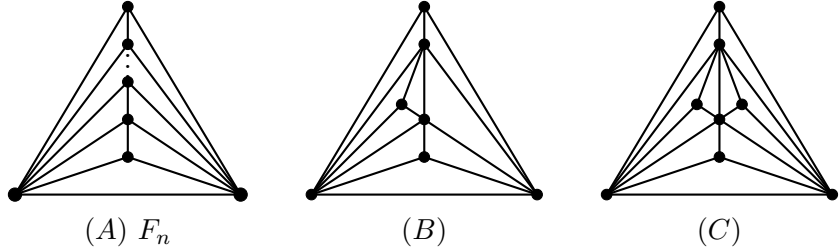
Thus, we have $f(n,C_3)=3n-8$ when $n\geq 6$ and $f(n,C_4)=\frac{1}{2}(n^2+3n-22)$ for $n\geq 5$. In~\cite{us}, the last five authors extended the results of Hakimi and Schmeichel by determining $f(n,C_5)$ for all~$n$. Asymptotic results for $f(n,C_{2k})$ with $k=3,4,5,6$, have recently been obtained by Cox and Martin~\cite{ccox,ccox2}.   

In the case when $H$ is a complete bipartite graph, Alon and Caro~\cite{alon}  determined the value of $f(n,H)$ exactly. They obtained the following results. 
\begin{theorem}[Alon and Caro~\cite{alon}]
\label{tm}
For all $k\geq 2$ and $n\geq 4$, \[f(n,K_{1,k})=2\binom{n-1}{k}+2\binom{3}{k}+(n-4)\binom{4}{k}.\]
\end{theorem}

\begin{theorem}[Alon and Caro~\cite{alon}]\label{tm2}

For all $k\geq 2$ and $n\geq 4$,

\begin{align*}
f(n,K_{2,k})= \begin{cases}
\binom{n-2}{k}, &\text{if $k\geq 5$ or $k=4$ and $n \neq 6$;}\\
3, &\text{if $(k,n)=(4,6)$;}\\
\binom{n-2}{3}, &\text{if $k=3$, $n\neq 6$;}\\
12, &\text{if $(k,n)=(3,6)$;}\\
\binom{n-2}{2}+4n-14, &\text{if $k=2$.}
\end{cases}
\end{align*}
\end{theorem}

Other results in this direction include a linear bound on the maximum number of copies of a $3$-connected planar graph by Wormald~\cite{Wormald} and independently Eppstein~\cite{eppstein}.  The exact bound on the maximum number of copies of $K_4$ was obtained independently by Alon and Caro~\cite{alon} and by Wood~\cite{k4}.
Let $P_k$ denote the path on $k$ vertices. It is well-known that $f(n,P_2)=3n-6$, and it follows from Theorem~\ref{tm} that $f(n,P_3)=n^2+3n-16$ for $n\geq 4$. 
The order of magnitude of $f(n, H)$ when $H$ is a fixed tree was determined in~\cite{gen} and for general $H$ (and in arbitrary surfaces) by Huynh,  Joret and Wood~\cite{surfaces} (see also~\cite{general} for results in general sparse settings). In particular, for a path on $k$ vertices, we have $f(n, P_k)= \Theta(n^{{\lfloor{\frac{k-1}{2}}\rfloor}+1})$. 

In this paper we determine $f(n,P_4)$,  the  maximum  number  of  copies  of a path on $4$ vertices possible  in  $n$-vertex  planar graph.  Our main result is the following.

\begin{theorem}\label{thm1}
We have,
\begin{align*}
f(n,P_4)= \begin{cases}
12, &\text{if $n=4$;}\\
147, &\text{if $n=7$;}\\
222, &\text{if $n=8$;}\\
7n^2-32n+27, &\text{if $n=5, 6$ and $n\geq 9$.}
\end{cases}
\end{align*}
For integers $n \in \{4,5,6\}$ and $n\geq 9$, the only $n$-vertex planar graph attaining the value $f(n, P_3)$ is the graph  $F_n$. For $n=7$ and $n=8$ the graphs pictured in Figure~\ref{qq}(B) and~\ref{qq}(C), respectively are the only graphs attaining the value $f(n, P_4)$. 
\label{main}
\end{theorem}
We also note that the asymptotic value of $f(n,P_5)$ was determined to be $n^3$ in~\cite{chacha}, and Cox and Martin~\cite{ccox} obtained the asymptotic result $f(n,P_7)=\frac{4}{27}n^4+O(n^{4-1/5}).$

\section{Notation and Preliminaries}
Let $G$ be a planar graph. We denote the vertex and the edge sets of $G$ by $V(G)$ and $E(G)$, respectively. For a vertex $v\in V(G)$, $d_G(v)$ denotes the degree of $v$. We omit the subscript whenever there is no ambiguity about which graph we are referring to. We denote by $N(v)$ the set of neighbors of $v$. For two vertices $x,y\in V(G)$, we denote the set of vertices which are adjacent to both $x$ and $y$ by $N(x,y)$. We also denote the size of $N(x,y)$ by $d(x,y)$. The minimum degree of $G$ is denoted by $\delta(G)$. For simplicity, we refer to a path of length three as a \emph{$3$-path}.  We denote the number of $P_4$'s in $G$ by $P_4(G)$. For $x\in V(G)$, the number of $P_4$'s in $G$ containing $x$ is denoted by $P_4(G,x)$. 
We denote the $n$-vertex graph obtained by joining every vertex from a path with $n-2$ vertices to two additional adjacent vertices by $F_n$ (pictured in Figure~\ref{qq}($A$)). 

For any maximal planar graph $G$ on $n$ vertices ($n \ge 3$) it can be shown that $3\leq \delta(G)\leq 5$. Moreover for a vertex $v$ in $V(G)$, if $d(v)=k$ and $N(v)=\{x_1,x_2,x_3,\dots ,x_k\}$, then $N(v)$ induces a unique cycle of length $k$. We may choose a drawing of $G$ so that $v$ is contained in the interior of the cycle. Without loss of generality, we may assume that we have a cycle $C$ with vertex sequence $x_1,x_2,x_3,\dots,x_k,x_1$. Let us denote the edge $\{x_i,v\}$ by $e_i$ for $i=1,2,3,\dots,k$ (see Figure~\ref{cycle}).       
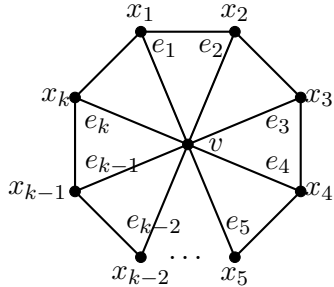
\begin{figure}[h]
\centering
\begin{tikzpicture}[scale=0.25]
\draw[fill=black](-6,-2.5)circle(8pt);
\draw[fill=black](-6,2.5)circle(8pt);
\draw[fill=black](-2.5,6)circle(8pt);
\draw[fill=black](2.5,6)circle(8pt);
\draw[fill=black](6,2.5)circle(8pt);
\draw[fill=black](6,-2.5)circle(8pt);
\draw[fill=black](-2.5,-6)circle(8pt);
\draw[fill=black](2.5,-6)circle(8pt);
\draw[fill=black](0,0)circle(8pt);
\draw (0,-6)node{$\dots$};
\draw[thick](-6,2.5)--(-6,2.5)--(-2.5,6)--(2.5,6)--(6,2.5)--(6,-2.5)--(2.5,-6)(-2.5,-6)--(-6,-2.5)--(-6,2.5);
\draw[thick](0,0)--(-6,-2.5)(0,0)--(-6,2.5) (0,0)--(-2.5,6)(0,0)--(2.5,6)(0,0)--(6,2.5)(0,0)--(6,-2.5)
(0,0)--(2.5,-6) (0,0)--(-2.5,-6);

\draw (0,0) -- (-2.5,6) node[above=8pt,midway]{$e_{1}$};
\draw (0,0) -- (2.5,6) node[above=8pt,midway]{$e_{2}$};
\node at (1.5,0){$v$};
\node at (-2.5,7){$x_1$};
\node at (2.5,7){$x_2$};
\node at (7,2.5){$x_3$};
\node at (7,-2.5){$x_4$};
\node at (2.5,-7){$x_5$};
\node at (-2.5,-7){$x_{k-2}$};
\node at (-8,-2.5){$x_{k-1}$};
\node at (-7,2.5){$x_{k}$};
\node at (2.7,-4.3){$e_{5}$};
\node at (-1.8,-4.3){$e_{k-2}$};
\node at (4.8,1.08){$e_{3}$};
\node at (4.8,-1.08){$e_{4}$};
\node at (-4.8,1.08){$e_{k}$};
\node at (-4,-1.08){$e_{k-1}$};
\end{tikzpicture}
\caption{Neighbors of a vertex $v\in V(G)$ of degree $k$.} 
\label{cycle}
\end{figure} 

We partition the set of $3$-paths containing $v$ into three different classes, depending on the location of their middle edge.

A \textbf{Type-I, $3$-path with respect to a vertex $v$} is a $3$-path which contains an edge $e_i$ as its middle edge (see Figure~\ref{type1}).
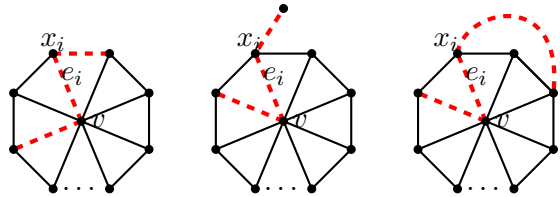
\begin{figure}[h]
\centering
\begin{tikzpicture}[scale=0.15]
\draw[dashed, red, ultra thick](0,0)--(-2.5,6)(0,0)--(-6,-2.5)(-2.5,6)--(2.5,6);
\draw[fill=black](-6,-2.5)circle(10pt);
\draw[fill=black](-6,2.5)circle(10pt);
\draw[fill=black](-2.5,6)circle(10pt);
\draw[fill=black](2.5,6)circle(10pt);
\draw[fill=black](6,2.5)circle(10pt);
\draw[fill=black](6,-2.5)circle(10pt);
\draw[fill=black](-2.5,-6)circle(10pt);
\draw[fill=black](2.5,-6)circle(10pt);
\draw[fill=black](0,0)circle(10pt);
\draw (0,-6)node{$\dots$};
\draw[thick](-6,2.5)--(-6,2.5)--(-2.5,6)(2.5,6)--(6,2.5)--(6,-2.5)--(2.5,-6)(-2.5,-6)--(-6,-2.5)--(-6,2.5);
\draw[thick](0,0)--(-6,2.5) (0,0)--(2.5,6)(0,0)--(6,2.5)(0,0)--(6,-2.5)
(0,0)--(2.5,-6) (0,0)--(-2.5,-6);

\node at (-0.8,4) {$e_i$};
\node at (1.5,0){$v$};
\node at (-2.5,7){$x_i$};

\end{tikzpicture}\qquad
\begin{tikzpicture}[scale=0.15]
\draw[dashed, red, ultra thick](0,0)--(-2.5,6)--(0,10)(0,0)--(-6,2.5);
\draw[fill=black](-6,-2.5)circle(10pt);
\draw[fill=black](-6,2.5)circle(10pt);
\draw[fill=black](-2.5,6)circle(10pt);
\draw[fill=black](2.5,6)circle(10pt);
\draw[fill=black](6,2.5)circle(10pt);
\draw[fill=black](6,-2.5)circle(10pt);
\draw[fill=black](-2.5,-6)circle(10pt);
\draw[fill=black](2.5,-6)circle(10pt);
\draw[fill=black](0,0)circle(10pt);
\draw (0,-6)node{$\dots$};
\draw[fill=black](0,10)circle(10pt);
\draw[thick](-6,2.5)--(-6,2.5)--(-2.5,6)--(2.5,6)--(6,2.5)--(6,-2.5)--(2.5,-6)(-2.5,-6)--(-6,-2.5)--(-6,2.5);
\draw[thick](0,0)--(-6,-2.5) (0,0)--(2.5,6)(0,0)--(6,2.5)(0,0)--(6,-2.5)
(0,0)--(2.5,-6) (0,0)--(-2.5,-6);

\node at (-0.8,4) {$e_i$};
\node at (1.5,0){$v$};
\node at (-3,7){$x_i$};
\end{tikzpicture}\qquad
\begin{tikzpicture}[scale=0.15]
\draw[dashed, red, ultra thick](0,0)--(-2.5,6)(0,0)--(-6,2.5);
\draw[dashed, red,ultra thick](-2.5,6)..controls (-2,10) and (7,12) .. (6,2.5);
\draw[fill=black](-6,-2.5)circle(10pt);
\draw[fill=black](-6,2.5)circle(10pt);
\draw[fill=black](-2.5,6)circle(10pt);
\draw[fill=black](2.5,6)circle(10pt);
\draw[fill=black](6,2.5)circle(10pt);
\draw[fill=black](6,-2.5)circle(10pt);
\draw[fill=black](-2.5,-6)circle(10pt);
\draw[fill=black](2.5,-6)circle(10pt);
\draw[fill=black](0,0)circle(10pt);
\draw (0,-6)node{$\dots$};
\draw[thick](-6,2.5)--(-6,2.5)--(-2.5,6)(2.5,6)--(6,2.5)--(6,-2.5)--(2.5,-6)(-2.5,-6)--(-6,-2.5)--(-6,2.5);
\draw[thick](-2.5,6)--(2.5,6);
\draw[thick](0,0)--(-6,-2.5) (0,0)--(2.5,6)(0,0)--(6,2.5)(0,0)--(6,-2.5)
(0,0)--(2.5,-6) (0,0)--(-2.5,-6)(6,2.5)--(2.5,6);

\node at (-0.8,4) {$e_i$};
\node at (1.5,0){$v$};
\node at (-3.5,7){$x_i$};
\end{tikzpicture}
\caption{Examples of Type-I, $3$-paths.}
\label{type1}
\end{figure}   

A \textbf{Type-II, $3$-path with respect to a vertex $v$} is a $3$-path which starts with vertices $v, x_i, x_j$. Furthermore, if the middle edge is an edge of the cycle $C$, then we call such a $3$-path a {Type-II(A), $3$-path.} Otherwise, we call it a Type-II(B), $3$-path (see Figure~\ref{type2}).

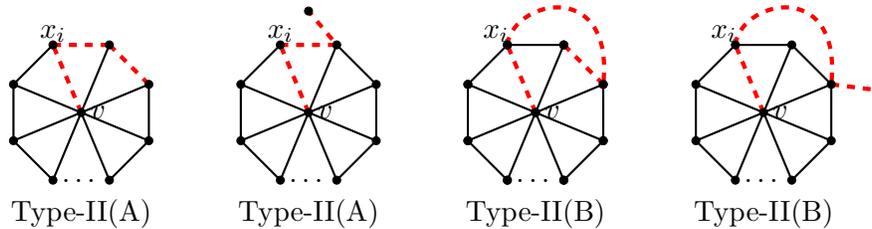
\begin{figure}[h]
\centering
\begin{tikzpicture}[scale=0.15]
\draw[dashed, red, ultra thick](0,0)--(-2.5,6)(-2.5,6)--(2.5,6)(2.5,6)--(6,2.5);
\draw[fill=black](-6,-2.5)circle(10pt);
\draw[fill=black](-6,2.5)circle(10pt);
\draw[fill=black](-2.5,6)circle(10pt);
\draw[fill=black](2.5,6)circle(10pt);
\draw[fill=black](6,2.5)circle(10pt);
\draw[fill=black](6,-2.5)circle(10pt);
\draw[fill=black](-2.5,-6)circle(10pt);
\draw[fill=black](2.5,-6)circle(10pt);
\draw[fill=black](0,0)circle(10pt);
\draw (0,-6)node{$\dots$};
\draw[thick](-6,2.5)--(-6,2.5)--(-2.5,6)(6,2.5)--(6,-2.5)--(2.5,-6)(-2.5,-6)--(-6,-2.5)--(-6,2.5);
\draw[thick](0,0)--(-6,-2.5) (0,0)--(2.5,6)(0,0)--(6,2.5)(0,0)--(6,-2.5) (0,0)--(2.5,-6) (0,0)--(-2.5,-6)(0,0)--(-6,2.5);
\node at (1.5,0){$v$};
\node at (-2.5,7){$x_i$};
\node at (0,-9){Type-II(A)};
\end{tikzpicture} \qquad
\begin{tikzpicture}[scale=0.15]
\draw[dashed, red, ultra thick](0,0)--(-2.5,6)(-2.5,6)--(2.5,6)--(0,9);
\draw[fill=black](-6,-2.5)circle(10pt);
\draw[fill=black](-6,2.5)circle(10pt);
\draw[fill=black](-2.5,6)circle(10pt);
\draw[fill=black](2.5,6)circle(10pt);
\draw[fill=black](6,2.5)circle(10pt);
\draw[fill=black](6,-2.5)circle(10pt);
\draw[fill=black](-2.5,-6)circle(10pt);
\draw[fill=black](2.5,-6)circle(10pt);
\draw[fill=black](0,0)circle(10pt);
\draw (0,-6)node{$\dots$};
\draw[fill=black](0,9)circle(10pt);
\draw[thick](-6,2.5)--(-6,2.5)--(-2.5,6)(2.5,6)--(6,2.5)--(6,-2.5)--(2.5,-6)(-2.5,-6)--(-6,-2.5)--(-6,2.5);
\draw[thick](0,0)--(-6,-2.5) (0,0)--(2.5,6)(0,0)--(6,2.5)(0,0)--(6,-2.5)
(0,0)--(2.5,-6) (0,0)--(-2.5,-6)(0,0)--(-6,2.5);
\node at (1.5,0){$v$};
\node at (-2.5,7){$x_i$};
\node at (0,-9){Type-II(A)};
\end{tikzpicture} \qquad
\begin{tikzpicture}[scale=0.15]
\draw[dashed, red,ultra thick](-2.5,6)..controls (-2,10) and (7,12)  .. (6,2.5);
\draw[dashed, red, ultra thick](0,0)--(-2.5,6);
\draw[fill=black](-6,-2.5)circle(10pt);
\draw[fill=black](-6,2.5)circle(10pt);
\draw[fill=black](-2.5,6)circle(10pt);
\draw[fill=black](2.5,6)circle(10pt);
\draw[fill=black](6,2.5)circle(10pt);
\draw[fill=black](6,-2.5)circle(10pt);
\draw[fill=black](-2.5,-6)circle(10pt);
\draw[fill=black](2.5,-6)circle(10pt);
\draw[fill=black](0,0)circle(10pt);
\draw (0,-6)node{$\dots$};
\draw[thick](-6,2.5)--(-6,2.5)--(-2.5,6)(6,2.5)--(6,-2.5)--(2.5,-6)(-2.5,-6)--(-6,-2.5)--(-6,2.5);
\draw[thick](-2.5,6)--(2.5,6);
\draw[thick](0,0)--(-6,-2.5) (0,0)--(2.5,6)(0,0)--(6,2.5)(0,0)--(6,-2.5)
(0,0)--(2.5,-6) (0,0)--(-2.5,-6)(0,0)--(-6,2.5);

\draw[fill=black](2.5,6)circle(6pt);
\draw[fill=black](6,2.5)circle(6pt);
\node at (1.5,0){$v$};
\node at (-3.5,7){$x_i$};
\node at (0,-9){Type-II(B)};
\draw[dashed, red, ultra thick] (6,2.5)--(2.5,6);
\draw[fill=black](-6,-2.5)circle(10pt);
\draw[fill=black](-6,2.5)circle(10pt);
\draw[fill=black](-2.5,6)circle(10pt);
\draw[fill=black](2.5,6)circle(10pt);
\draw[fill=black](6,2.5)circle(10pt);
\draw[fill=black](6,-2.5)circle(10pt);
\draw[fill=black](-2.5,-6)circle(10pt);
\draw[fill=black](2.5,-6)circle(10pt);
\draw[fill=black](0,0)circle(10pt);
\end{tikzpicture} \qquad
\begin{tikzpicture}[scale=0.15]
\draw[dashed, red, ultra thick](0,0)--(-2.5,6)(6,2.5)--(10,2);
\draw[dashed, red,ultra thick](-2.5,6)..controls (-2,10) and (7,12)  .. (6,2.5);
\draw[fill=black](-6,-2.5)circle(10pt);
\draw[fill=black](-6,2.5)circle(10pt);
\draw[fill=black](-2.5,6)circle(10pt);
\draw[fill=black](2.5,6)circle(10pt);
\draw[fill=black](6,2.5)circle(10pt);
\draw[fill=black](6,-2.5)circle(10pt);
\draw[fill=black](-2.5,-6)circle(10pt);
\draw[fill=black](2.5,-6)circle(10pt);
\draw[fill=black](0,0)circle(10pt);
\draw (0,-6)node{$\dots$};
\draw[fill=black](10,2)circle(10pt);
\draw[thick](-6,2.5)--(-6,2.5)--(-2.5,6)(2.5,6)--(6,2.5)--(6,-2.5)--(2.5,-6)(-2.5,-6)--(-6,-2.5)--(-6,2.5);
\draw[thick](-2.5,6)--(2.5,6);
\draw[thick](0,0)--(-6,-2.5) (0,0)--(2.5,6)(0,0)--(6,2.5)(0,0)--(6,-2.5)
(0,0)--(2.5,-6) (0,0)--(-2.5,-6)(0,0)--(-6,2.5);

\draw[fill=black](10,2)circle(0.8pt);
\draw[fill=black](6,2.5)circle(6pt);
\node at (1.5,0){$v$};
\node at (-3.5,7){$x_i$};
\node at (0,-9){Type-II(B)};
\end{tikzpicture}
\caption{Examples of Type-II, $3$-paths.}
\label{type2}
\end{figure}

A \textbf{Type-III, $3$-path with respect to a vertex $v$} is a $3$-path which starts at the vertex $v$ such that its middle edge connects a vertex from $N(v)$ to a vertex from $V(G)\setminus (N(v) \cup \{v\})$. Furthermore, if the last vertex is not from $N(v)$, then we call such a $3$-path a Type-III(A), $3$-path. Otherwise, we call it a Type-III(B), $3$-path (see Figure~\ref{type3}). 

\begin{figure}[h]
\centering
\begin{tikzpicture}[scale=0.15]
\draw[dashed, red, ultra thick](0,0)--(-2.5,6)--(4,9.6)--(5,6);
\draw[fill=black](-6,-2.5)circle(10pt);
\draw[fill=black](-6,2.5)circle(10pt);
\draw[fill=black](-2.5,6)circle(10pt);
\draw[fill=black](2.5,6)circle(10pt);
\draw[fill=black](6,2.5)circle(10pt);
\draw[fill=black](6,-2.5)circle(10pt);
\draw[fill=black](-2.5,-6)circle(10pt);
\draw[fill=black](2.5,-6)circle(10pt);
\draw[fill=black](0,0)circle(10pt);
\draw (0,-6)node{$\dots$};
\draw[fill=black](4,9.6)circle(10pt);
\draw[fill=black](5,6)circle(10pt);
\draw[thick](-6,2.5)--(-6,2.5)--(-2.5,6)--(2.5,6)--(6,2.5)--(6,-2.5)--(2.5,-6)(-2.5,-6)--(-6,-2.5)--(-6,2.5);
\draw[thick](0,0)--(-6,-2.5) (0,0)--(2.5,6)(0,0)--(6,2.5)(0,0)--(6,-2.5)
(0,0)--(2.5,-6) (0,0)--(-2.5,-6)(0,0)--(-6,2.5);
\node at (1.5,0){$v$};
\node at (-2.5,7){$x_1$};
\node at(0,-9){Type-III(A)};
\end{tikzpicture}\qquad
\begin{tikzpicture}[scale=0.15]
\draw[dashed, red, ultra thick](0,0)--(-2.5,6)--(4,9.6)--(6,2.5);
\draw[fill=black](-6,-2.5)circle(10pt);
\draw[fill=black](-6,2.5)circle(10pt);
\draw[fill=black](-2.5,6)circle(10pt);
\draw[fill=black](2.5,6)circle(10pt);
\draw[fill=black](6,2.5)circle(10pt);
\draw[fill=black](6,-2.5)circle(10pt);
\draw[fill=black](-2.5,-6)circle(10pt);
\draw[fill=black](2.5,-6)circle(10pt);
\draw[fill=black](0,0)circle(10pt);
\draw (0,-6)node{$\dots$};
\draw[fill=black](4,9.6)circle(10pt);
\draw[thick](-6,2.5)--(-6,2.5)--(-2.5,6)--(2.5,6)--(6,2.5)--(6,-2.5)--(2.5,-6)(-2.5,-6)--(-6,-2.5)--(-6,2.5);
\draw[thick](0,0)--(-6,-2.5) (0,0)--(2.5,6)(0,0)--(6,2.5)(0,0)--(6,-2.5)
(0,0)--(2.5,-6) (0,0)--(-2.5,-6)(0,0)--(-6,2.5);
\node at (1.5,0){$v$};
\node at (-2.5,7){$x_1$};
\node at(0,-9){Type-III(B)};
\end{tikzpicture}
\caption{Examples of Type-III, $3$-paths.}
\label{type3}
\end{figure}
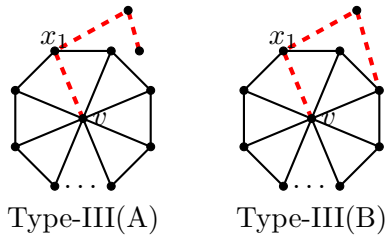

It is easy to see that each of the $3$-paths containing the vertex $v$ is in exactly one of the three classes which we have defined. For simplicity, we will sometimes write Type-(I), (II), (III), {$3$-path} instead of Type-(I), (II), (III), $3$-path with respect to a vertex $v$, when the vertex under consideration is clear.

We will use the following two lemmas in our proof of the main theorem.  The first lemma gives the number of $3$-paths in a given graph $G$. 
\begin{lemma}\label{lemma1}
For a graph $G$, the number of $3$-paths in $G$ is
$$P_4(G)=\sum\limits_{\{x,y\}\in E(G)}(d(x)-1)(d(y)-1)-3\N(C_3,G).$$
\end{lemma}
\begin{proof}
Consider an edge $\{x,y\}\in E(G)$ and count the number of 3-paths containing $x$ as the second and and $y$ the third vertex of the 3-path. There are $d(x)-1$ possibilities to choose the first vertex and $d(y)-1$ possibilities to choose the last vertex of the path. Since the first and the last vertex of the 3-path need to be different, from the total number of $(d(x)-1)(d(y)-1)$ possibilities we need to subtract the number of triangles containing the edge $\{x,y\}$, which is $d(x,y)$. 

Therefore,
$$P_4(G) = \sum_{\{x,y\}\in E(G)}\left((d(x)-1)(d(y)-1)-d(x,y)\right) = \sum_{\{x,y\}\in E(G)}(d(x)-1)(d(y)-1)-3\N(C_3,G),$$
as each triangle is counted 3 times in the sum.
This completes the proof of Lemma~\ref{lemma1}.
\end{proof}

With this lemma we can prove the following lemma.
\begin{lemma} \label{lemma:d>3}
For every $n$-vertex planar graph $G$ with $\delta(G)\geq 4$ we have   
\[
P_4(G)< 7n^2-36n+50.
\]
\end{lemma}
\begin{proof}
Without loss of generality we may assume that $G$ is a maximal planar graph with $3n-6$ edges and $2n-4$ triangular faces. In particular it contains at least $2n-4$ triangles. 

From Lemma~\ref{lemma1} the total number of 3-paths in $G$ is equal to
\begin{align*}
    P_4(G) &= \sum_{\{x,y\}\in E(G)} (d(x)-1)(d(y)-1)-3\N(C_3,G)\\
    &\leq \sum_{\{x,y\}\in E(G)} (d(x)-1)(d(y)-1)-3(2n-4)\\
    &= \frac{1}{2}\sum_{x\in V(G)}
     (d(x)-1) \left( \sum_{y\in N(x)}  d(y)-d(x) \right)-6n+12.
\end{align*}

Since $\delta(G)\geq 4$ and the sum of the degrees of all the vertices is equal to $2e(G)=6n-12$, for each vertex $x$ we have
\[
\sum_{y\in N(x)}d(y)=6n-12-d(x)-\sum_{y\notin N[x]}d(y)\leq 6n-12-d(x)-4(n-1-d(x))=3d(x)+2n-8.
\]

This gives us the following bound
\begin{align*}
     P_4(G) &\leq \frac{1}{2}\sum_{x\in V(G)} (d(x)-1) \left( 2d(x)+2n-8 \right)-6n+12 \\ 
     &= \sum_{x\in V(G)}d^2(x)+(n-5)\sum_{x\in V(G)}d(x)-n(n-4)-6n+12\\
     &\leq \left((n-1)^2+(n-3)^2+4^2(n-2)\right)+(n-5)(6n-12) -n^2-2n+12\\
     &= 7n^2-36n+50,
\end{align*}
where the last inequality comes from convexity.

It remains to notice that, since $\delta(G)\geq 4$ and $G$ is a planar graph, we have $n\geq 6$, hence $7n^2-36n+50<7n^2-32n+27$.
\end{proof}

\section{Proof of Theorem~\ref{thm1}}
We are going to prove the theorem by induction on the number of vertices. 
The base cases, when $n\leq 9$, will be discussed later. 

Let $G$ be a planar graph on $n$ vertices. Then we have $3\leq \delta(G)\leq 5$.
From Lemma~\ref{lemma:d>3} we may assume $\delta(G)=3$. We are going to prove the rest by induction, after removing a vertex of degree $3$.


Let $v$ be a vertex of degree $3$ and $N(v)=\{x_1, x_2, x_3\}$. 
Our goal is to show that $P_4(G,v) \leq 14n-39$. 
Indeed, by deleting the vertex $v$ we obtain a maximal planar graph $G'$ on $(n-1)$ vertices, and by the induction hypothesis we have $P_4(G') \leq 7(n-1)^2-32(n-1)+27$. 
Therefore, \[P_4(G) \leq 7(n-1)^2-32(n-1)+27+14n-39=7n^2-32n+27.\]  

Notice that the vertices $x_1,x_2,x_3$ induce a triangle. Denote the edges $\{x_i, v\}$ by $e_i$, $i\in\{1,2,3\}$. The number of Type-I, $3$-paths with $e_i$ in the middle is $2d(x_i)-4$ for all $i\in\{1,2,3\}$. Thus we have $2\sum_{i=1}^3d(x_i)-12$
Type-I, $3$-paths. The number of Type-II, $3$-paths  starting at $v$ and continuing to a vertex $x_i$, $i\in\{1,2,3\}$, is $d(x_1)+d(x_2)+d(x_3) - d(x_i) -4$. Thus, we have $2\sum_{i=1}^3d(x_i)-12$ Type-II, $3$-paths. It remains to count the number of Type-III, $3$-paths with respect to  the vertex~$v$. For this we need to consider two subcases.

\subsubsection*{Case 1.1: $N(x_1)\cap N(x_2)\cap N(x_3)=\{v\}$.}
For each edge $e$ which is not incident to the triangle, we can have at most four Type-III(A), $3$-paths with respect to the vertex $v$ (see Figure~\ref{4type}). Since there are at most $(3n-6)-\Big(\sum_{i=1}^3 d(x_i)-3\Big)$ such edges which are not incident to the triangle, it follows that the number of Type-III(A), $3$-paths  is at most 
$4\Big(3n-6-\big(\sum_{i=1}^3 d(x_i)-3\big)\Big).$

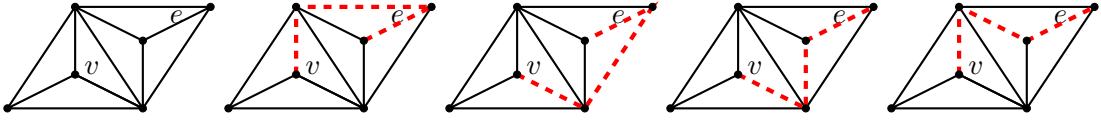
\begin{figure}[h]
\centering
\begin{tikzpicture}[scale=0.45]
\draw[fill=black] (0,0) circle (3pt);
\draw[fill=black] (4,0) circle (3pt);
\draw[fill=black] (2,1) circle (3pt);
\draw[fill=black] (2,3) circle (3pt);
\draw[fill=black] (6,3) circle (3pt);
\draw[fill=black] (4,2) circle (3pt);
\node at (2.5,1.2) {$v$};
\node at (5,2.7) {$e$};
\draw[thick] (0,0) -- (4,0) -- (2,1)--(0,0);
\draw[thick] (2,3) -- (4,0) -- (2,1)--(2,3);
\draw[thick](0,0) -- (2,3);
\draw[black,thick]((4,0)--(4,2)--(6,3)--(2,3)--(4,2)(6,3)--(4,0);
\end{tikzpicture}
\begin{tikzpicture}[scale=0.45]
\node at (2.5,1.2) {$v$};
\node at (5,2.7) {$e$};
\draw[thick] (0,0) -- (4,0) -- (2,1)--(0,0);
\draw[thick] (2,3) -- (4,0) -- (2,1);
\draw[thick](0,0) -- (2,3);
\draw[black,thick](4,0)--(4,2)(2,3)--(4,2)(6,3)--(4,0);
\draw[dashed, red,ultra thick](2,1)--(2,3)--(6,3)--(4,2);
\draw[fill=black] (0,0) circle (3pt);
\draw[fill=black] (4,0) circle (3pt);
\draw[fill=black] (2,1) circle (3pt);
\draw[fill=black] (2,3) circle (3pt);
\draw[fill=black] (6,3) circle (3pt);
\draw[fill=black] (4,2) circle (3pt);
\end{tikzpicture}
\begin{tikzpicture}[scale=0.45]
\node at (2.5,1.2) {$v$};
\node at (5,2.7) {$e$};
\draw[thick] (0,0) -- (4,0)  (2,1)--(0,0);
\draw[thick] (2,3) -- (4,0) (2,1)--(2,3);
\draw[thick](0,0) -- (2,3);
\draw[black,thick](4,0)--(4,2)(6,3)--(2,3)--(4,2);
\draw[dashed, red,ultra thick](2,1)--(4,0)--(6,3)--(4,2);
\draw[fill=black] (0,0) circle (3pt);
\draw[fill=black] (4,0) circle (3pt);
\draw[fill=black] (2,1) circle (3pt);
\draw[fill=black] (2,3) circle (3pt);
\draw[fill=black] (6,3) circle (3pt);
\draw[fill=black] (4,2) circle (3pt);
\end{tikzpicture}
\begin{tikzpicture}[scale=0.45]
\node at (2.5,1.2) {$v$};
\node at (5,2.7) {$e$};
\draw[thick] (0,0) -- (4,0) (2,1)--(0,0);
\draw[thick] (2,3) -- (4,0) (2,1)--(2,3);
\draw[thick](0,0) -- (2,3);
\draw[black,thick](6,3)--(2,3)--(4,2)(6,3)--(4,0);
\draw[dashed, red,ultra thick](2,1)--(4,0)--(4,2)--(6,3);
\draw[fill=black] (0,0) circle (3pt);
\draw[fill=black] (4,0) circle (3pt);
\draw[fill=black] (2,1) circle (3pt);
\draw[fill=black] (2,3) circle (3pt);
\draw[fill=black] (6,3) circle (3pt);
\draw[fill=black] (4,2) circle (3pt);
\end{tikzpicture}
\begin{tikzpicture}[scale=0.45]
\node at (2.5,1.2) {$v$};
\node at (5,2.7) {$e$};
\draw[thick] (0,0) -- (4,0) -- (2,1)--(0,0);
\draw[thick] (2,3) -- (4,0) -- (2,1);
\draw[thick](0,0) -- (2,3);
\draw[black,thick](4,0)--(4,2)(6,3)--(2,3)(6,3)--(4,0);
\draw[dashed, red,ultra thick](2,1)--(2,3)--(4,2)--(6,3);
\draw[fill=black] (0,0) circle (3pt);
\draw[fill=black] (4,0) circle (3pt);
\draw[fill=black] (2,1) circle (3pt);
\draw[fill=black] (2,3) circle (3pt);
\draw[fill=black] (6,3) circle (3pt);
\draw[fill=black] (4,2) circle (3pt);
\end{tikzpicture}
\caption{Four Type-III(A), $3$-paths for a fixed edge $e$.}
\label{4type}
\end{figure}
The remaining $3$-paths are Type-III(B), $3$-paths. Recall that in this case each vertex $v'\neq v$ can be adjacent to at most 2 vertices of the triangle induced by $N(v)$. Thus for each such vertex $v'$, $v'\notin \{x_1, x_2, x_3, v\}$, we have at most two Type-III(B), $3$-paths (see Figure~\ref{2type}). Hence we have at most $2(n-4)$ Type-III(B), $3$-paths. Thus we get,
\[P_4(G,v)\leq 4\sum_{i=1}^3d(x_i)-24+4\Big(3n-3-\sum_{i=1}^3d(x_i)\Big)+2(n-4)=14n-44.\]
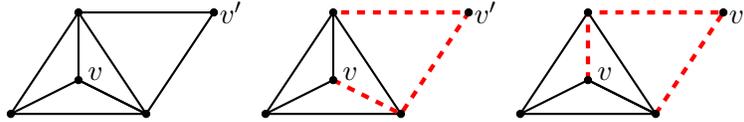
\begin{figure}[h]
\centering
\begin{tikzpicture}[scale=0.45]
\node at (2.5,1.2) {$v$};
\node at (6.5,3) {$v'$};
\draw[thick] (0,0) -- (4,0) -- (2,1)--(0,0);
\draw[thick] (2,3) -- (4,0) -- (2,1)--(2,3);
\draw[thick](0,0) -- (2,3);
\draw[black,thick]((4,0)--(6,3)--(2,3);
\draw[fill=black] (0,0) circle (3pt);
\draw[fill=black] (4,0) circle (3pt);
\draw[fill=black] (2,1) circle (3pt);
\draw[fill=black] (2,3) circle (3pt);
\draw[fill=black] (6,3) circle (3pt);
\end{tikzpicture}
\begin{tikzpicture}[scale=0.45]
\node at (2.5,1.2) {$v$};
\node at (6.5,3) {$v'$};
\draw[thick] (0,0) -- (4,0) (2,1)--(0,0);
\draw[thick] (2,3) -- (4,0) (2,1)--(2,3);
\draw[thick](0,0) -- (2,3);
\draw[dashed, red,ultra thick](2,1)--(4,0)--(6,3)--(2,3);
\draw[fill=black] (0,0) circle (3pt);
\draw[fill=black] (4,0) circle (3pt);
\draw[fill=black] (2,1) circle (3pt);
\draw[fill=black] (2,3) circle (3pt);
\draw[fill=black] (6,3) circle (3pt);
\end{tikzpicture}
\begin{tikzpicture}[scale=0.45]
\node at (2.5,1.2) {$v$};
\node at (6.5,3) {$v'$};
\draw[thick] (0,0) -- (4,0) -- (2,1)--(0,0);
\draw[thick] (2,3) -- (4,0) -- (2,1);
\draw[thick](0,0) -- (2,3);
\draw[dashed, red,ultra thick](4,0)--(6,3)--(2,3)--(2,1);
\draw[fill=black] (0,0) circle (3pt);
\draw[fill=black] (4,0) circle (3pt);
\draw[fill=black] (2,1) circle (3pt);
\draw[fill=black] (2,3) circle (3pt);
\draw[fill=black] (6,3) circle (3pt);
\end{tikzpicture}
\caption{Type-III(B), $3$-paths for a fixed vertex $v'$, $v'\notin \{v, x_1, x_2, x_3\}$.}
\label{2type}
\end{figure}

Therefore, $P_4(G,v)\leq 14n-44 < 14n-39$ and we have no extremal graph in this case. 

\subsubsection*{Case 1.2: There exists a vertex $u$, $u\neq v$, such that $N(x_1)\cap N(x_2)\cap N(x_3)=\{v,u\}$.}
We consider the three regions formed by vertices $u, x_1, x_2$ and $x_3$. Let the region defined by the vertices $u, x_1$ and $x_2$ which does not contain $x_3$ be $R_1$, the region defined by the vertices $u,x_2$ and $x_3$ which does not contain $x_1$ be $R_2$, and lastly the region defined by the vertices $u,x_1$ and $x_3$ and not containing $x_2$ be $R_3$, as shown in Figure~\ref{regions}.
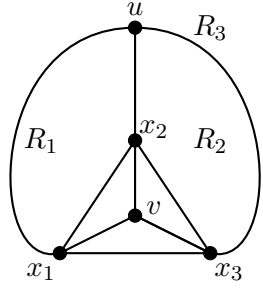
\begin{figure}[h]
\centering
\begin{tikzpicture}[scale=0.5]
\draw[fill=black] (0,0) circle (5pt);
\draw[fill=black] (4,0) circle (5pt);
\draw[fill=black] (2,1) circle (5pt);
\draw[fill=black] (2,3) circle (5pt);
\draw[fill=black] (2,6) circle (5pt);
\node at (-0.5,-0.5) {$x_1$};
\node at (4.5,-0.5) {$x_3$};
\node at (2.5,1.2) {$v$};
\node at (2.5,3.3) {$x_2$};
\node at (2,6.5) {$u$};
\node at (-0.5,3) {$R_1$};
\node at (4,3) {$R_2$};
\node at (4,6) {$R_3$};
\draw[thick] (0,0) -- (4,0) -- (2,1)--(0,0);
\draw[thick] (2,3) -- (4,0) -- (2,1)--(2,3);
\draw[thick](0,0) -- (2,3);
\draw[thick](2,3)--(2,6);
\draw[black,thick](0,0)..controls (-2,-0.5) and (-2,6) .. (2,6);
\draw[black,thick](4,0)..controls (6,-0.5) and (6,6) .. (2,6);
\end{tikzpicture}
\caption{Three regions formed by the vertex $u$ and the vertices of the triangle.}
\label{regions}
\end{figure}

From the planarity of $G$, notice that there is at most one edge $e_1$ with end vertices $u$ and $y_1$ such that $y_1$ lies inside the region $R_1$ and $y_1$ is adjacent to both $x_1$ and $x_2$. Similarly there is at most one edge $e_2$ and $e_3$ with respect to the regions $R_2$ and $R_3$ respectively meeting the conditions stated for $e_1$. We refer to the edges $e_1$, $e_2$ and $e_3$ as \textit{star edges} of $G$ with respect to the vertex $v$.

Take an edge $e$ such that $V(e) \cap \{x_1, x_2, x_3\}=\emptyset$. Then there are at most five  Type-III(A), $3$-paths with respect to the vertex $v$, containing the edge $e$, since $G$ is planar. Furthermore, for each star edge (if one exists) in the three regions there are five Type-III(A), $3$-paths. 
Figure~\ref{fivepaths} shows an edge $e$ in region $R_1$ with all five possible $3$-paths of this kind.

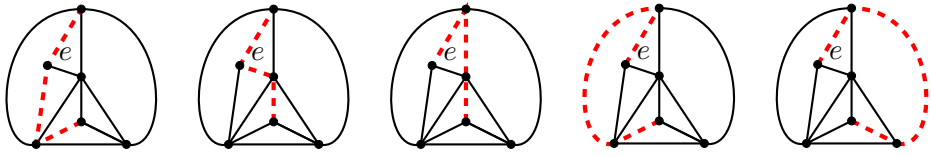
\begin{figure}[h]
\centering
\begin{tikzpicture}[scale=0.3]
\draw[thick] (0,0) -- (4,0) -- (2,1);
\draw[thick] (2,3) -- (4,0) -- (2,1)--(2,3);
\draw[thick](0,0) -- (2,3);
\draw[thick](2,3)--(2,6);
\draw[black,thick](0,0)..controls (-2,-0.5) and (-2,6) .. (2,6);
\draw[black,thick](4,0)..controls (6,-0.5) and (6,6) .. (2,6);
\draw[black,thick](0.5,3.5)--(2,3);
\draw[dashed, red,ultra thick](2,1)--(0,0)--(0.5,3.5)--(2,6);
\node at (1.3,4){$e$};
\draw[fill=black] (0,0) circle (5pt);
\draw[fill=black] (4,0) circle (5pt);
\draw[fill=black] (2,1) circle (5pt);
\draw[fill=black] (2,3) circle (5pt);
\draw[fill=black] (2,6) circle (5pt);
\draw[fill=black] (0.5,3.5) circle (5pt);
\end{tikzpicture}
\begin{tikzpicture}[scale=0.3]
\draw[thick] (0,0) -- (4,0) -- (2,1)--(0,0);
\draw[thick] (2,3) -- (4,0) -- (2,1);
\draw[thick](0,0) -- (2,3);
\draw[thick](2,3)--(2,6);
\draw[black,thick](0,0)..controls (-2,-0.5) and (-2,6) .. (2,6);
\draw[black,thick](4,0)..controls (6,-0.5) and (6,6) .. (2,6);
\draw[black,thick](0,0)--(0.5,3.5);
\draw[dashed, red,ultra thick](2,1)--(2,3)--(0.5,3.5)--(2,6);
\node at (1.3,4){$e$};
\draw[fill=black] (0.5,3.5) circle (5pt);
\draw[fill=black] (0,0) circle (5pt);
\draw[fill=black] (4,0) circle (5pt);
\draw[fill=black] (2,1) circle (5pt);
\draw[fill=black] (2,3) circle (5pt);
\draw[fill=black] (2,6) circle (5pt);
\end{tikzpicture}
\begin{tikzpicture}[scale=0.3]
\draw[thick] (0,0) -- (4,0) -- (2,1)--(0,0);
\draw[thick] (2,3) -- (4,0) -- (2,1);
\draw[thick](0,0) -- (2,3);
\draw[black,thick](0,0)..controls (-2,-0.5) and (-2,6) .. (2,6);
\draw[black,thick](4,0)..controls (6,-0.5) and (6,6) .. (2,6);
\draw[black,thick](0,0)--(0.5,3.5)--(2,3);
\draw[dashed, red,ultra thick](2,1)--(2,3)--(2,6)--(0.5,3.5);
\node at (1.3,4){$e$};
\draw[fill=black] (0,0) circle (5pt);
\draw[fill=black] (4,0) circle (5pt);
\draw[fill=black] (2,1) circle (5pt);
\draw[fill=black] (2,3) circle (5pt);
\draw[fill=black] (2,6) circle (5pt);
\draw[fill=black] (0.5,3.5) circle (5pt);
\end{tikzpicture}
\begin{tikzpicture}[scale=0.3]
\draw[thick] (0,0) -- (4,0)--  (2,1);
\draw[thick] (2,3) -- (4,0) -- (2,1)--(2,3);
\draw[thick](0,0) -- (2,3);
\draw[thick](2,3)--(2,6);
\draw[dashed, red,ultra thick](0,0)..controls (-2,-0.5) and (-2,6) .. (2,6);
\draw[black,thick](4,0)..controls (6,-0.5) and (6,6) .. (2,6);
\draw[black,thick](0,0)--(0.5,3.5)--(2,3);
\draw[dashed, red,ultra thick](2,1)--(0,0)(0.5,3.5)--(2,6);
\node at (1.3,4){$e$};
\draw[fill=black] (0,0) circle (5pt);
\draw[fill=black] (4,0) circle (5pt);
\draw[fill=black] (2,1) circle (5pt);
\draw[fill=black] (2,3) circle (5pt);
\draw[fill=black] (2,6) circle (5pt);
\draw[fill=black] (0.5,3.5) circle (5pt);
\end{tikzpicture}
\begin{tikzpicture}[scale=0.3]
\draw[thick] (0,0) -- (4,0)  (2,1)--(0,0);
\draw[thick] (2,3) -- (4,0)  (2,1)--(2,3);
\draw[thick](0,0) -- (2,3);
\draw[thick](2,3)--(2,6);
\draw[black,thick](0,0)..controls (-2,-0.5) and (-2,6) .. (2,6);
\draw[dashed, red, ultra thick](4,0)..controls (6,-0.5) and (6,6) .. (2,6);
\draw[black,thick](0,0)--(0.5,3.5)--(2,3);
\draw[dashed, red,ultra thick](2,1)--(4,0)(0.5,3.5)--(2,6);
\draw[fill=black] (0,0) circle (5pt);
\draw[fill=black] (4,0) circle (5pt);
\draw[fill=black] (2,1) circle (5pt);
\draw[fill=black] (2,3) circle (5pt);
\draw[fill=black] (2,6) circle (5pt);
\draw[fill=black] (0.5,3.5) circle (5pt);
\node at (1.3,4){$e$};
\end{tikzpicture}
\caption{Five Type-III(A), $3$-paths that contains the star edge $e$.}
\label{fivepaths}
\end{figure}\ \\
Notice that for each vertex $w$ inside the regions, one can have at most two Type-III(B), $3$-paths containing $w$. For the vertex $u$, we have six Type-III(B), $3$-paths containing the vertex $u$ (see Figure~\ref{sixpath}).

\begin{figure}[h]
\centering
\begin{tikzpicture}[scale=0.3]
\draw[thick] (0,0) -- (4,0) (2,1)--(0,0);
\draw[thick] (2,3) -- (4,0) (2,1)--(2,3);
\draw[thick](0,0) -- (2,3);
\draw[black,thick](0,0)..controls (-2,-0.5) and (-2,6) .. (2,6);
\draw[dashed, red,ultra thick](4,0)..controls (6,-0.5) and (6,6) .. (2,6);
\draw[dashed, red,ultra thick](4,0)--(2,1)(2,6)--(2,3);
\draw[fill=black] (0,0) circle (5pt);
\draw[fill=black] (4,0) circle (5pt);
\draw[fill=black] (2,1) circle (5pt);
\draw[fill=black] (2,3) circle (5pt);
\draw[fill=black] (2,6) circle (5pt);
\end{tikzpicture}  
\begin{tikzpicture}[scale=0.3]
\draw[thick] (0,0) -- (4,0)-- (2,1)--(0,0);
\draw[thick] (2,3) -- (4,0)--(2,1);
\draw[thick](0,0) -- (2,3);
\draw[black,thick](0,0)..controls (-2,-0.5) and (-2,6) .. (2,6);
\draw[dashed, red,ultra thick](4,0)..controls (6,-0.5) and (6,6) .. (2,6);
\draw[dashed, red,ultra thick](2,3)--(2,1)(2,6)--(2,3);
\draw[fill=black] (0,0) circle (5pt);
\draw[fill=black] (4,0) circle (5pt);
\draw[fill=black] (2,1) circle (5pt);
\draw[fill=black] (2,3) circle (5pt);
\draw[fill=black] (2,6) circle (5pt);
\end{tikzpicture}  
\begin{tikzpicture}[scale=0.3]
\draw[thick] (0,0) -- (4,0) -- (2,1)--(0,0);
\draw[thick] (2,3) -- (4,0) -- (2,1);
\draw[thick](0,0) -- (2,3);
\draw[dashed, red,ultra thick](0,0)..controls (-2,-0.5) and (-2,6) .. (2,6);
\draw[black,thick](4,0)..controls (6,-0.5) and (6,6) .. (2,6);
\draw[dashed, red,ultra thick](2,3)--(2,1)(2,6)--(2,3);
\draw[fill=black] (0,0) circle (5pt);
\draw[fill=black] (4,0) circle (5pt);
\draw[fill=black] (2,1) circle (5pt);
\draw[fill=black] (2,3) circle (5pt);
\draw[fill=black] (2,6) circle (5pt);
\end{tikzpicture}  
\begin{tikzpicture}[scale=0.3]
\draw[thick] (0,0) -- (4,0) --  (2,1);
\draw[thick] (2,3) -- (4,0)  (2,1)--(2,3);
\draw[thick](0,0) -- (2,3);
\draw[dashed, red,ultra thick](0,0)..controls (-2,-0.5) and (-2,6) .. (2,6);
\draw[black,thick](4,0)..controls (6,-0.5) and (6,6) .. (2,6);
\draw[dashed, red,ultra thick](0,0)--(2,1)(2,6)--(2,3);
\draw[fill=black] (0,0) circle (5pt);
\draw[fill=black] (4,0) circle (5pt);
\draw[fill=black] (2,1) circle (5pt);
\draw[fill=black] (2,3) circle (5pt);
\draw[fill=black] (2,6) circle (5pt);
\end{tikzpicture}  
\begin{tikzpicture}[scale=0.3]
\draw[thick] (0,0) -- (4,0) -- (2,1);
\draw[thick] (2,3) -- (4,0) -- (2,1)--(2,3);
\draw[thick](0,0) -- (2,3);
\draw[thick](2,3)--(2,6);
\draw[dashed, red,ultra thick](0,0)..controls (-2,-0.5) and (-2,6) .. (2,6);
\draw[dashed, red,ultra thick](4,0)..controls (6,-0.5) and (6,6) .. (2,6);
\draw[dashed,red,ultra thick](2,1)--(0,0);
\draw[fill=black] (0,0) circle (5pt);
\draw[fill=black] (4,0) circle (5pt);
\draw[fill=black] (2,1) circle (5pt);
\draw[fill=black] (2,3) circle (5pt);
\draw[fill=black] (2,6) circle (5pt);
\end{tikzpicture}  
\begin{tikzpicture}[scale=0.3]
\draw[thick] (0,0) -- (4,0)  (2,1)--(0,0);
\draw[thick] (2,3) -- (4,0)  (2,1)--(2,3);
\draw[thick](0,0) -- (2,3);
\draw[thick](2,3)--(2,6);
\draw[dashed,red,ultra thick](0,0)..controls (-2,-0.5) and (-2,6) .. (2,6);
\draw[dashed,red,ultra thick](4,0)..controls (6,-0.5) and (6,6) .. (2,6);
\draw[dashed,red,ultra thick](2,1)--(4,0);
\draw[fill=black] (0,0) circle (5pt);
\draw[fill=black] (4,0) circle (5pt);
\draw[fill=black] (2,1) circle (5pt);
\draw[fill=black] (2,3) circle (5pt);
\draw[fill=black] (2,6) circle (5pt);
\end{tikzpicture}    
\caption{Six Type-III(B), $3$-paths with respect to the vertex $v$.}
\label{sixpath}
\end{figure}
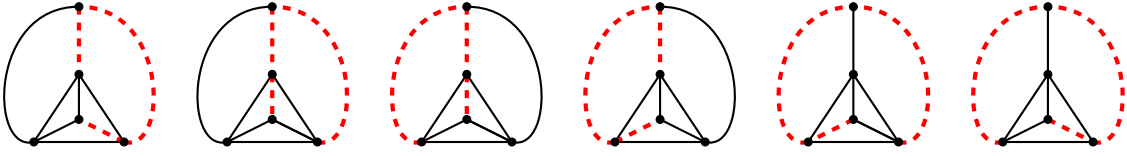

\begin{enumerate}[label=\arabic*.,font=\bfseries]
\item If there is no star edge in each of the three regions, then we have at most 
\[4\Big(3n-6-(\sum_{i=1}^3d(x_i)-3)\Big)+2(n-5)+6=-4\sum_{i=1}^3d(x_i)+14n-16\] Type-III, $3$-paths containing  the vertex $v$. Thus, we have 
\[
P_4(G,v)\leq 4\sum_{i=1}^3d(x_i)-24-4\sum_{i=1}^3d(x_i)+14n-16  \leq 14n-40.\]
Therefore, $P_4(G,v)<14n-39.$   

\item If there is only one star edge, then we have at most 
\[4\Big(3n-6-(\sum_{i=1}^3d(x_i)-3)-1\Big)+5+2(n-5)+6=-4\sum_{i=1}^3d(x_i)+14n-15\]
Type-III, $3$-paths with respect to the vertex $v$. Therefore $P_4(G,v)\leq 14n-39$.

\begin{remark}\label{equality_triangle}
Equality holds  if we have a vertex $v$ of degree three and a vertex $u$ which is adjacent to all of the vertices incident to $v$.  All the other vertices share exactly $2$ neighbors with $v$, and we have exactly one star edge. 
\end{remark}

\item If there are exactly two star edges, then we have two regions containing them. Without loss of generality, let the regions be $R_1$ and $R_2$. The third region, $R_3$, may or may not contain a vertex.
\begin{enumerate}[label*=\arabic*]
\item  If there is a vertex in $R_3$, then at least one vertex in $R_3$ is a neighbor of the vertex $u$, hence this vertex is not a neighbor of at least one of the vertices $x_1$ or $x_3$. Otherwise, we would have another star edge. It follows that there is no Type-III(B), $3$-path containing this vertex. Thus, the number of Type-III, $3$-paths with respect to the vertex $v$ is at most
\[
4\Big(3n-6-(\sum_{i=1}^3d(x_i)-3)-2\Big)+10+2(n-6)+6=-4\sum_{i=1}^3d(x_i)+14n-16. 
\]
So we have $P_4(G,v)\leq 14n-40$. Therefore, $P_4(G,v)<14n-39.$   
\item  If there is no vertex in the region $R_3$, then at least one of the regions $R_1$ or $R_2$ contains at least two vertices, since $n \geq 10$. Without loss of generality, suppose $R_1$ contains at least two vertices. Let $f_1$ be the star edge in the region. This edge is incident to $u$, and we denote the other vertex it is incident to by $u_1$. We have $u_1\in N(x_1)\cap N(x_2)$. If there is a vertex in the region $R$, defined by the vertices $x_1,u_1$ and $u$ not containing $x_2$, then there is an edge $\{u_1, u_1'\}$ in the region $R$, where $u_1' \notin \{x_1, x_2, x_3\}$. This edge is in at most three Type-III(A), $3$-paths. Moreover $u_1'$ is not incident to the vertex $x_2$. Hence $u_1'$ is not in any of the Type-III(B) paths. Therefore we have at most \[4\Big(3n-6-(\sum_{i=1}^3d(x_i)-3)-2-1\Big)+13+2(n-6)+6=-4\sum_{i=1}^3d(x_i)+14n-17\]
Type-III, $3$-paths with respect to the vertex $v$. Consequently, we have $P_4(G,v)\leq 14n-41$. Therefore, $P_4(G,v)<14n-39.$ 

Similarly the region defined by the vertices $x_2,u_1,u$ not containing $x_1$ is also empty, otherwise we are done by induction.

 Thus the vertices must be in the region $R'$, defined by the vertices $x_1,x_2,u_1$ not containing $u$. Consider an edge $f_2=\{u_1,u_2\}$ in the region $R'$. If $u_2$ is the only vertex in the region $R'$, then $N(u_2)=\{x_1, x_2, u_1\}$, and we are done by induction, since we have a vertex $u_2$ of degree three with at most one star edge, which was settled in Cases~1.2(1) and~1.2(2) (see Figure~\ref{specialvertex}).

\begin{figure}[h]
\centering
\begin{tikzpicture}[scale=0.8]
\draw[thick] (0,0) -- (4,0) -- (2,1)--(0,0);
\draw[thick] (2,3) -- (4,0) -- (2,1)--(2,3);
\draw[thick](2,3)--(2,6);
\draw[black, thick](0,0)..controls (-2,-0.5) and (-2,6) .. (2,6);
\draw[black,thick](4,0)..controls (6,-0.5) and (6,6) .. (2,6);
\draw[black,thick](2,3)(0.5,3.5)--(2,6);
\draw[dashed, red,ultra thick](0,0)--(2,3)--(0.5,3.5)--(1,2.7)--(2,3)(0,0)--(0.5,3.5)--(1,2.7)--(0,0);

\filldraw[fill=black, fill opacity=0.15] plot coordinates {(2,3) (2,6) (0.5,3.5)} ;
\path [fill=black, fill opacity=0.15 ] (0,0)..controls (-2,-0.5) and (-2,6) .. (2,6)--(0.5,3.5)--(0,0);
\draw[black, thick](2,6)--(0.5,3.5)(2,6)--(2,3);
\node at (-0.5,-0.5) {$x_1$};
\node at (4.5,-0.5) {$x_3$};
\node at (2.5,1.2) {$v$};
\node at (2.5,3.3) {$x_2$};
\node at (2,6.3) {$u$};
\node at (1.5,4.7){$f_1$};
\node at (0.2,3.5){$u_1$};
\node at (1.25,2.5){$u_2$};
\node at (0.9,3.2){$f_2$};
\draw[fill=black] (0,0) circle (3pt);
\draw[fill=black] (4,0) circle (3pt);
\draw[fill=black] (2,1) circle (3pt);
\draw[fill=black] (2,3) circle (3pt);
\draw[fill=black] (2,6) circle (3pt);
\draw[fill=black] (0.5,3.5) circle (3pt);
\draw[fill=black] (1,2.7) circle (3pt);

\end{tikzpicture}
\caption{A vertex $u_2$ with the property that two of the corresponding regions have no vertex inside.}
\label{specialvertex}
\end{figure}
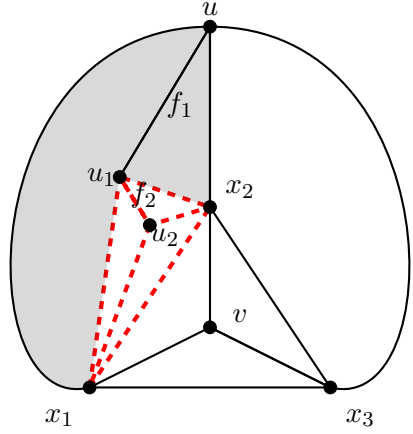

If the vertex $u_2$ is not a neighbor of one of the vertices $x_1$ or  $x_2$, then the edge $f_2$ is not incident to the triangle and is contained in at most three Type-III(A), $3$-paths. Moreover $u_2$ is in none of the Type-III(B) paths. Therefore, we have at most \[4\Big(3n-6-(\sum_{i=1}^3d(x_i)-3\Big)-2-1)+13+2(n-6)+6=-4\sum_{i=1}^3d(x_i)+14n-17\] Type-III, $3$-paths. Consequently, we have $P_4(G,v)\leq 14n-41$. Therefore, $P_4(G,v)<14n-39.$ 

We have that there are at least two vertices in the region $R'$, and $u_2$ is incident with both of the vertices  $x_1$ and $x_2$.

A similar argument to the one given in Case 1.2(3.1) gives us that there is no vertex in the region defined by the vertices $x_1, u_2, u_1$  not containing $x_2$ and, likewise, in the region defined by the vertices  $x_2, u_2, u_1$ not containing $x_1$. Thus, all the vertices must be in the region defined by the vertices $x_1, u_2, x_2$ not containing $u_1$; let us denote this region by $R''$.
Consider an edge $f_3=\{u_2,u_3\}$ in the region $R''$. 
Thus we proceed with a similar argument as before, this time applied to the region $R''$ and the corresponding vertex $u_3$. 
Notice that $N(u_3)=\{x_1,x_2,u_2\}$. 
If $u_3$ is the only vertex in $R''$, then we are done by induction since $u_3$ would be a vertex of degree three and with at most one star edge, which was settled in Case~1.2(1) and Case~1.2(2). 
Otherwise, we get a region containing at least one vertex, say $R'''$, defined by the vertices $x_1,x_2$ and $u_3$ not containing $u_2$. 
We apply similar reasoning to $R'''$ as applied for $R'$ and $R''$. 
Since $G$ is finite, after a finite number of steps $k$, we obtain a vertex $u_k$, such that $N(u_k)=\{x_1,x_2,u_{k-1}\}$ and  $u_k$ is incident with at most one star edge, which was settled in Case~1.2(1) and Case~1.2(2).   
\end{enumerate}

\item Suppose there are three star edges. Let $\{u,y_i\}$ be the star edge in the region $R_i$, for each $i \in\{1,2,3\}$. Since $n\geq 10$, one of the regions $R_1, R_2$ or $R_3$ contains at least one additional vertex other than $y_i$. Without loss of generality, let $R_1$ be such a region. If there is a vertex in the region $x_1, y_1, u$ not containing $x_2$, then we have at least one edge, say $\{y_1,y_1'\} $, for some $y_1'$ inside the region bounded by $x_1,y_1$ and $u$ not containing $x_2$. The edge $\{y_1,y_1'\}$ is in at most three Type-III(A), $3$-paths. Moreover, the vertex $y_1'$ is not incident to $x_2$ and $x_3$. Hence it is not in any Type-III(B) paths. Thus, we have at most \[4\Big(3n-6-(\sum_{i=1}^3d(x_i)-3)-4\Big)+18+2(n-6)+6=-4\sum_{i=1}^{3}d(x_i)+14n-16\]
Type-III, $3$-paths. 
Hence we have  $P_4(G,v)\leq 14n-40$. Therefore, $P_4(G,v)<14n-39.$ 

Similarly, the region defined by $x_2$, $y_1$ and $u$ not containing the vertex $x_1$ must be empty. Otherwise, we are done by induction. 

If the region obtained from the vertices $x_1,y_1, x_2$ not containing $u$ contains only one vertex $u'$, then we have a degree three vertex $u'$, and there is at most one star edge corresponding to the vertex $u'$. Hence, we are done by induction as in Case 1.2(1) or Case 1.2(2) for the vertex $u'$. Otherwise, if the region obtained by the vertices $x_1,y_1, x_2$ not containing $u$ contains more than one vertex, then we are done by similar arguments given in Case 1.2(3.2).
\end{enumerate}

\subsection*{Basis for the induction}
Here we are going to find the maximum number of paths of length three in a planar graph with at most $9$ vertices. This will form the basis for the induction. We are going to recall some facts from the previous calculations. Let  $G$ be a maximal planar graph on $n$ vertices, and $v\in V(G)$ be a vertex of minimum degree. 

If $d(v)=3$, then we have the following.
\begin{itemize}
    \item Suppose there is no vertex other than $v$ adjacent to all the neighbors of $v$, then from Case~1.1 we have 
    \begin{align}
        P_4(G,v)\leq 14n-44.
        \label{0}
    \end{align}
   
    \item Suppose there is a vertex other than $v$ which is adjacent to all the neighbors of $v$, then we consider the following cases. 
    \begin{itemize}
        \item If there is no star edge  with respect to the vertex $v$, then from Case 1.2.1 we have
        \begin{align}
            P_4(G,v)\leq 14n-40.
            \label{1}
        \end{align}
        \item  If there is only one star edge with respect to the vertex $v$, then from Case 1.2.2 we have
        \begin{align}
            P_4(G,v)\leq 14n-39.
            \label{12}
        \end{align}
       
        \item If there are two star edges with respect to the vertex $v$, then in this case we cannot use Case~1.2.3, since $n$ is not at least $10$. However, by similar calculations we have a weaker result for all $n$.
        \begin{align}
             \begin{split}
                  P_4(G,v)&\leq 4\sum_{i=1}^3d(x_i)-24+4\Big(3n-6-(\sum_{i=1}^3d(x_i)-3)-2\Big)+10+2(n-5)+6\\&=14n-38.
             \end{split}
            \label{13}
        \end{align}
        \item If there are three star edges with respect to the vertex $v$, then in this case we cannot use Case~1.2.4, since  $n\leq 9$. However, by similar calculations we have a weaker result for all~$n$. 
        
        \begin{align}
        \begin{split}
        P_4(G,v)&\leq 4\sum_{i=1}^3d(x_i)-24+4\Big(3n-6-(\sum_{i=1}^3d(x_i)-3)-3\Big)+15+2(n-5)+6\\&=14n-37.  
        \end{split}
        \label{14}
        \end{align}
        \end{itemize}
\end{itemize}

\begin{claim} $f(4,P_4)=12$ and $f(5,P_4)=42$.
\end{claim} 
\begin{proof}
The maximal planar graphs with 4 and 5 vertices are unique. The graphs are $K_4$ and $K_5^-$ respectively.
It is easy to verify that $f(4,P_4)=12$ and $f(5,P_4)=42$.
\end{proof}
\begin{claim} $f(6,P_4)=87$.\label{c1} \end{claim}
\begin{proof}
Let $G$ be a maximal planar graph on $6$ vertices. We have $\delta(G)=3$. First we prove the following claim. 
\begin{claim}\label{cm3}
There is a vertex different from $v$ which is adjacent to every neighbor of $v$. Moreover, there is one star edge with respect to $v$. 
\end{claim}
\begin{proof}
Let $N(v)=\{x_1,x_2,x_3\}$ and the remaining two vertices other than $v,x_1,x_2$ and $x_3$ be $y_1$ and $y_2$. By maximality, every edge of $G$ must be incident to exactly two triangular faces. Thus each of the edges in $\{x_1x_2, x_2x_3,x_3x_1\}$  must be incident to a triangular face which is not incident to $v$. The number of vertices contained in the triangular region bounded by $x_1,x_2$ and $x_3$ not containing $v$ is 2, namely $y_1$ and $y_2$. Thus, two of the edges, say $x_1x_2$ and $x_2x_3$, must use one of the vertices in $\{y_1,y_2\}$, say $y_1$, such that the $x_1x_2y_1$ and $x_2x_3y_1$ are triangular faces incident to the two edges. From the property that every face of a maximal planar graph is of size 3, necessarily $y_1$ and $y_2$ must be adjacent. Hence we obtain that the vertex $y_1$ is adjacent to every neighbor of $v$. Moreover, the edge $y_1y_2$ is the only star edge of $G$ with respect to $v$. 
\end{proof}
Now we proceed in proving Claim~\ref{c1}. Deleting the vertex $v$, we get a maximal planar graph on $5$ vertices which contains $42$ $3$-paths. 

Therefore, using Claim~\ref{cm3} we have that the number of $3$-paths which contain the vertex $v$ is at most 45,  from \eqref{1} and \eqref{12}. Thus $P_4(G)\leq 42+45 =87$ and we have a unique extremal graph $F_6$ with $87$ $3$-paths. 
\end{proof}


\begin{claim}\label{c2}  $f(7,P_4)=147$. \end{claim}
\begin{proof}
Let $G$ be a maximal planar graph on $7$ vertices. We have  $d(v)=3$. Deleting this vertex we get a maximal planar graph with $6$ vertices and containing at most $87$ $3$-paths. Since the number of vertices is $7$, there are at most two star edges. Therefore using \eqref{0}, \eqref{1}, \eqref{12} and \eqref{13}, the maximum number of $3$-paths containing the vertex is 60. Hence $P_4(G)\leq 147$ and equality holds if we deleted a vertex with two star edges and the graph we obtained was $F_6$. There are only two faces in $F_6$ where we can place the deleted vertex in order to have two star edges, in both cases we get the same graph which pictured in Figure~\ref{qq}(B).
\end{proof}



\begin{claim}\label{c3}  $f(8,P_4)=222$.\end{claim}
\begin{proof}
Let $G$ be a maximal planar graph on $8$ vertices. Since $d(v)=3$,   after deleting the vertex $v$ from $G$, we get a seven vertex maximal planar graph containing at most $147$ paths of length three. However, from \eqref{0}, \eqref{1}, \eqref{12}, \eqref{13} and \eqref{14}, the maximum number of $3$-paths that contain the vertex $v$ is at most 75. Thus $P_4(G)\leq 222$ and equality holds if we have deleted a vertex with incident three star edges and the graph we got was also extremal (Figure~\ref{qq}(B)). There is a unique face of the graph in Figure~\ref{qq}(B) where we can place a deleted vertex in order to have three star edges. This leads us to the unique extremal graph pictured in Figure~\ref{qq}(C).
\end{proof}
\begin{claim}  $f(9,P_4)=306$.\end{claim}
\begin{proof}

If  there is no other vertex incident to all the vertices incident to the vertex $v$, then using~\eqref{0} we have at most 82 $3$-paths that contain $v$. Since deleting the vertex $v$ results in an eight vertex maximal planar graph, it contains at most $222$ $3$-paths, from Claim~\ref{c3}. Thus we have  $P_4(G)\leq 304$.

Now assume that the neighbors of $v$ have a common adjacent vertex other than $v$. Consider the three regions obtained as in Figure~\ref{regions}.
\begin{enumerate}
\item [(i)]If each of the three regions is nonempty, then there is a unique maximal planar graph of this kind (see Figure~\ref{16}).  Using Lemma~\ref{lemma1} one can compute that this planar graph contains $303$ $3$-paths.
\begin{figure}[h]
\centering
\begin{tikzpicture}[scale=0.2]
\draw[dashed, red,thick](-6,0)--(0,6)--(6,0)--(-6,0);
\draw[thick](0,3)--(-6,0) (0,3)--(0,6) (0,3)--(6,0);
\draw[thick](-6,0)--(0,-4)--(6,0);
\draw[thick](-6,0)--(-4,5)--(0,6);
\draw[dashed, red,ultra thick](-6,0)--(0,6)--(6,0)--(-6,0);
\draw[thick](-6,0)--(-5,6.5)--(0,6) (-4,5)--(-5,6.5) (-5,6.5)--(0,12);
\draw[thick](0,6)--(0,12);
\draw[thick](6,0)--(5,6.5)--(0,12);
\draw[thick](0,6)--(5,6.5);
\draw[black,thick](0,12)..controls (-9,11) and (-7,3) .. (-6,0);
\draw[black,thick](0,12)..controls (9,11) and (7,3) .. (6,0);
\draw[black,thick](0,12)..controls (12,11) and (12,-3) .. (0,-4);
\draw[fill=black](-6,0)circle(12pt);
\draw[fill=black](6,0)circle(12pt);
\draw[fill=black](0,6)circle(12pt);
\draw[fill=black](0,3)circle(12pt);
\draw[fill=black](0,12)circle(12pt);
\draw[fill=black](0,-4)circle(12pt);
\draw[fill=black](-4,5)circle(12pt);
\draw[fill=black](-5,6.5)circle(12pt);
\draw[fill=black](5,6.5)circle(12pt);
\node at (1,3.3){$v$};
\end{tikzpicture}
\caption{A maximal planar graph on $9$ vertices, containing 303 $3$-paths.}
\label{16}
\end{figure}
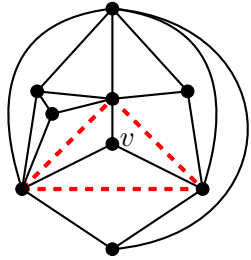
\item[(ii)] If two of the regions contain two vertices each, then the remaining region contains no vertex. The two nonempty regions contain a star edge. 

If in each of the two regions, we have a vertex which is incident to exactly one vertex of the triangle $N(v)$, then we have at most \[4\sum_{i=1}^3d(x_i)-24+4\Big(3n-6-(\sum_{i=1}^3d(x_i)-3)-2-2\Big)+16+2(n-7)+6=14n-44=82\] 
$3$-paths that contain the vertex $v$. Since deleting the vertex $v$ results in an eight vertex maximal planar graph, which contains at most $222$, $3$-paths, from  Claim~\ref{c3}, we get  $P_4(G)\leq 304$.

If only one of the two regions contain a vertex which is incident to exactly one vertex of the triangle, then there are only two such maximal planar graphs (see Figure~\ref{18}). The number of $3$-paths they contain are, respectively, 290 and 297.  
\begin{figure}[h]
\centering
\begin{tikzpicture}[scale=0.2]
\draw[dashed, red,thick](-6,0)--(0,6)--(6,0)--(-6,0);
\draw[thick](0,3)--(-6,0) (0,3)--(0,6) (0,3)--(6,0);
\draw[thick](-6,0)--(-4,5)--(0,6);
\draw[thick](6,0)--(5,6.5)--(0,12);
\draw[dashed, red,ultra thick](-6,0)--(0,6)--(6,0)--(-6,0);
\draw[thick](-6,0)--(-5,6.5)(0,6)--(-5,6.5) (-4,5)--(-5,6.5) (-5,6.5)--(0,12)(4,5)--(5,6.5);
\draw[thick](0,6)--(0,12);
\draw[thick](6,0)--(4,5)--(0,12);
\draw[thick](0,6)--(4,5);
\draw[black,thick](0,12)..controls (-9,11) and (-7,3) .. (-6,0);
\draw[black,thick](0,12)..controls (9,11) and (7,3) .. (6,0);
\node at (0,-2){$P_4(G)=290$};
\draw[fill=black](-6,0)circle(12pt);
\draw[fill=black](6,0)circle(12pt);
\draw[fill=black](0,6)circle(12pt);
\draw[fill=black](0,3)circle(12pt);
\draw[fill=black](0,12)circle(12pt);
\draw[fill=black](-4,5)circle(12pt);
\draw[fill=black](-5,6.5)circle(12pt);
\draw[fill=black](4,5)circle(12pt);
\draw[fill=black](5,6.5)circle(12pt);
\node at (1,3.3){$v$};
\end{tikzpicture}
\begin{tikzpicture}[scale=0.2]
\draw[dashed, red,ultra thick](-6,0)--(0,6)--(6,0)--(-6,0);
\draw[thick](0,3)--(-6,0) (0,3)--(0,6) (0,3)--(6,0);
\draw[thick](-6,0)--(-4,5)--(0,6);
\draw[dashed, red,ultra thick](-6,0)--(0,6)--(6,0)--(-6,0);
\draw[thick](6,0)--(5,6.5)--(0,12)(0,6)--(5,6.5);
\draw[thick](-6,0)--(-5,6.5)(0,6)--(-5,6.5) (-4,5)--(-5,6.5) (-5,6.5)--(0,12);
\draw[thick](0,6)--(0,12)(0,6)--(1.8,8)--(0,12)(1.8,8)--(5,6.5);
\draw[black,thick](0,12)..controls (-9,11) and (-7,3) .. (-6,0);
\draw[black,thick](0,12)..controls (9,11) and (7,3) .. (6,0);
\node at (0,-2){$P_4(G)=297$};
\draw[fill=black](-6,0)circle(12pt);
\draw[fill=black](6,0)circle(12pt);
\draw[fill=black](0,6)circle(12pt);
\draw[fill=black](0,3)circle(12pt);
\draw[fill=black](0,12)circle(12pt);
\draw[fill=black](-4,5)circle(12pt);
\draw[fill=black](-5,6.5)circle(12pt);
\draw[fill=black](1.8,8)circle(12pt);
\draw[fill=black](5,6.5)circle(12pt);
\node at (1,3.3){$v$};
\end{tikzpicture}
\caption{Maximal planar graphs on $9$ vertices.}
\label{18}
\end{figure}
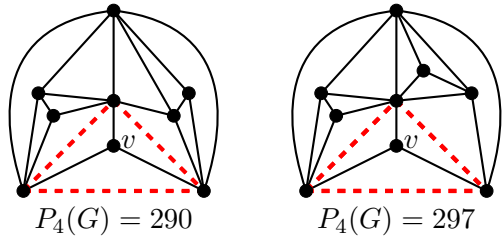

If in each of the two regions there is no vertex incident to exactly one vertex of the triangle, then the planar graph is unique (see Figure~\ref{17}). The number of $3$-paths in this graph is 296.
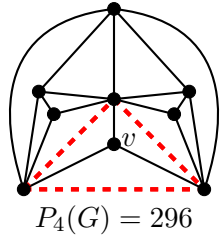
\begin{figure}[h]
\centering
\begin{tikzpicture}[scale=0.2]
\draw[dashed, red,ultra thick](-6,0)--(0,6)--(6,0)--(-6,0);
\draw[thick](0,3)--(-6,0) (0,3)--(0,6) (0,3)--(6,0);
\draw[thick](-6,0)--(-4,5)--(0,6);
\draw[thick](6,0)--(4,5)--(0,6);
\draw[dashed, red,ultra thick](-6,0)--(0,6)--(6,0)--(-6,0);
\draw[thick](-6,0)--(-5,6.5)--(0,6) (-4,5)--(-5,6.5) (-5,6.5)--(0,12)(4,5)--(5,6.5);
\draw[thick](0,6)--(0,12);
\draw[thick](6,0)--(5,6.5)--(0,12);
\draw[thick](0,6)--(5,6.5);
\draw[black,thick](0,12)..controls (-9,11) and (-7,3) .. (-6,0);
\draw[black,thick](0,12)..controls (9,11) and (7,3) .. (6,0);
\node at (0,-2){$P_4(G)=296$};
\draw[fill=black](-6,0)circle(12pt);
\draw[fill=black](6,0)circle(12pt);
\draw[fill=black](0,6)circle(12pt);
\draw[fill=black](0,3)circle(12pt);
\draw[fill=black](0,12)circle(12pt);
\draw[fill=black](-4,5)circle(12pt);
\draw[fill=black](-5,6.5)circle(12pt);
\draw[fill=black](4,5)circle(12pt);
\draw[fill=black](5,6.5)circle(12pt);
\node at (1,3.3){$v$};
\end{tikzpicture}
\caption{A maximal planar graph on $9$ vertices.} 
\label{17}
\end{figure}
\item[(iii)] Assume one of the regions contains a vertex and another contains three vertices (the third one is empty).

Suppose there is only one star edge, then the number of $3$-paths that contain the vertex $v$ is at most 
\[4\sum_{i=1}^3d(x_i)-24+4\Big(3n-6-(\sum_{i=1}^3d(x_i)-3)-1-1\Big)+8+2(n-6)+6=14n-42=84.\] 
Since one of the vertices of the triangle will be of degree $4$, after removing the vertex $v$, we will not have the unique extremal graph in Figure~\ref{qq}(C), since it does not contain a vertex of degree four. Thus, in this case, we have $P_4(G)<222+84=306$. 
After removing the vertex $v$ we will not get the unique extremal graph in Figure~\ref{qq}(C). Thus, in this case, we have $P_4(G)<222+84=306$. 

If there are two star edges, and there are two vertices  which are incident to exactly one of the vertices of the triangle, then we have at most \[4\sum_{i=1}^3d(x_i)-24+4\Big(3n-6-(\sum_{i=1}^3d(x_i)-3)-2-2\Big)+10+6+2(n-7)+6=14n-44=82\] $3$-paths containing the vertex $v$. Therefore $P_4(G)\leq 304$. 

If there are two star edges, and there are at least three vertices incident to two of the vertices of the triangle, then Figure~\ref{ddd} shows all possible nine vertex planar graphs. There are 300, 289, 292, 299 and 302, $3$-paths in those graphs, respectively.
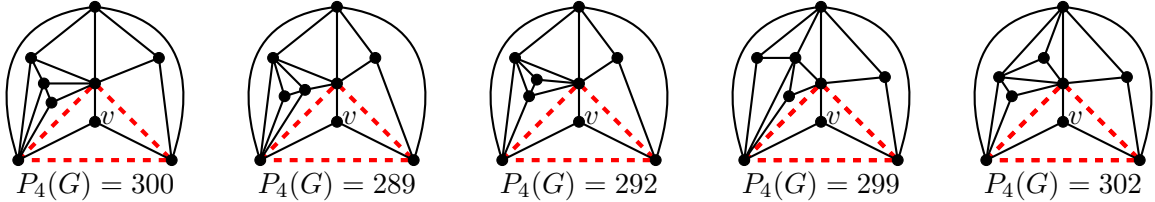
\begin{figure}[h]
\centering
\begin{tikzpicture}[scale=0.17] 
\draw[dashed, red,ultra thick](-6,0)--(0,6)--(6,0)--(-6,0);
\draw[thick](0,3)--(-6,0) (0,3)--(0,6) (0,3)--(6,0);
\draw[thick](6,0)--(5,8)--(0,6); 
\draw[thick](0,6)--(0,12)(-2.5,5.5);
\draw[thick](-6,0)--(-5,8)--(0,12);
\draw[thick](-6,0)--(-4,6)--(0,6);
\draw[thick](-6,0)--(-3.4,4.5)--(0,6);
\draw[thick](0,6)--(-5,8)--(-4,6)--(-3.4,4.5) (5,8) -- (0,12);
\draw[black,thick](0,12)..controls (-9,11) and (-7,3) .. (-6,0);
\draw[black,thick](0,12)..controls (9,11) and (7,3) .. (6,0);
\node at (0,-2){$P_4(G)=300$};
\draw[fill=black](-6,0)circle(12pt);
\draw[fill=black](6,0)circle(12pt);
\draw[fill=black](0,6)circle(12pt);
\draw[fill=black](0,3)circle(12pt);
\draw[fill=black](0,12)circle(12pt);
\draw[fill=black](-4,6)circle(12pt);
\draw[fill=black](-5,8)circle(12pt);
\draw[fill=black](-3.4,4.5)circle(12pt);
\draw[fill=black](5,8)circle(12pt);
\node at (1,3.3){$v$};
\end{tikzpicture}
\begin{tikzpicture}[scale=0.17]
\draw[dashed, red,ultra thick](-6,0)--(0,6)--(6,0)--(-6,0);
\draw[thick](0,3)--(-6,0) (0,3)--(0,6) (0,3)--(6,0);
\draw[thick](-6,0)--(-2.5,5.5)--(0,6); 
\draw[thick](0,6)--(0,12)(-2.5,5.5);
\draw[thick](-6,0)--(-5,8)--(0,12);
\draw[thick](0,12)--(3,8)--(6,0)(3,8)--(0,6);
\draw[thick](-5,8)--(-4.1,5)--(-6,0) (-2.5,5.5)--(-4.1,5);
\draw[thick](0,6)--(-5,8)--(-4,7)--(-3.2,6.2)--(-2.5,5.5);
\draw[black,thick](0,12)..controls (-9,11) and (-7,3) .. (-6,0);
\draw[black,thick](0,12)..controls (9,11) and (7,3) .. (6,0);
\node at (0,-2){$P_4(G)=289$};
\draw[fill=black](-6,0)circle(12pt);
\draw[fill=black](6,0)circle(12pt);
\draw[fill=black](0,6)circle(12pt);
\draw[fill=black](0,3)circle(12pt);
\draw[fill=black](0,12)circle(12pt);
\draw[fill=black](-4.1,5)circle(12pt);
\draw[fill=black](-5,8)circle(12pt);
\draw[fill=black](3,8)circle(12pt);
\draw[fill=black](-2.5,5.5)circle(12pt);
\node at (1,3.3){$v$};
\end{tikzpicture}
\begin{tikzpicture}[scale=0.17]
\draw[dashed, red,ultra thick](-6,0)--(0,6)--(6,0)--(-6,0);
\draw[thick](0,3)--(-6,0) (0,3)--(0,6) (0,3)--(6,0);
\draw[thick](0,6)--(0,12)(-2.5,5.5);
\draw[thick](-6,0)--(-5,8)--(0,12);
\draw[thick](0,12)--(3,8)--(6,0)(3,8)--(0,6);
\draw[thick](-5,8)--(-3.9,5)--(-6,0)(0,6)--(-3.9,5);
\draw[thick](0,6)--(-5,8)(-5,8)--(-3.3,6.3)--(0,6)(-3.9,5)--(-3.3,6.3);
\draw[black,thick](0,12)..controls (-9,11) and (-7,3) .. (-6,0);
\draw[black,thick](0,12)..controls (9,11) and (7,3) .. (6,0);
\node at (0,-2){$P_4(G)=292$};
\draw[fill=black](-6,0)circle(12pt);
\draw[fill=black](6,0)circle(12pt);
\draw[fill=black](0,6)circle(12pt);
\draw[fill=black](0,3)circle(12pt);
\draw[fill=black](0,12)circle(12pt);
\draw[fill=black](3,8)circle(12pt);
\draw[fill=black](-3.3,6.3)circle(12pt);
\draw[fill=black](-3.9,5)circle(12pt);
\draw[fill=black](-5,8)circle(12pt);
\node at (1,3.3){$v$};
\end{tikzpicture}
\begin{tikzpicture}[scale=0.17]
\draw[dashed, red,ultra thick](-6,0)--(0,6)--(6,0)--(-6,0);
\draw[dashed, red,ultra thick](-6,0)--(0,6)--(6,0)--(-6,0);
\draw[thick](0,3)--(-6,0) (0,3)--(0,6) (0,3)--(6,0);
\draw[thick](-6,0)--(-2.5,5)--(0,6) (-6,0)--(-2,8)--(0,6); 
\draw[thick](0,6)--(0,12)(-2.5,5)--(-2,8)--(0,12);
\draw[thick](6,0)--(5,6.5)--(0,12)(-6,0)--(-5,8)--(0,12)(-2,8)--(-5,8);
\draw[thick](0,6)--(5,6.5);
\draw[black,thick](0,12)..controls (-9,11) and (-7,3) .. (-6,0);
\draw[black,thick](0,12)..controls (9,11) and (7,3) .. (6,0);
\node at (0,-2){$P_4(G)=299$};
\draw[fill=black](-6,0)circle(12pt);
\draw[fill=black](6,0)circle(12pt);
\draw[fill=black](0,6)circle(12pt);
\draw[fill=black](0,3)circle(12pt);
\draw[fill=black](0,12)circle(12pt);
\draw[fill=black](-2.5,5)circle(12pt);
\draw[fill=black](5,6.5)circle(12pt);
\draw[fill=black](-2,8)circle(12pt);
\draw[fill=black](-5,8)circle(12pt);
\node at (1,3.3){$v$};
\end{tikzpicture}
\begin{tikzpicture}[scale=0.17]
\draw[dashed, red,ultra thick](-6,0)--(0,6)--(6,0)--(-6,0);
\draw[dashed, red,ultra thick](-6,0)--(0,6)--(6,0)--(-6,0);
\draw[thick](0,3)--(-6,0) (0,3)--(0,6) (0,3)--(6,0);
\draw[thick](-6,0)--(-4,5)--(0,6);
\draw[thick](-6,0)--(-5,6.5)--(0,6) (-4,5)--(-5,6.5) (-5,6.5)--(0,12)(-5,6.5)--(-1.5,8)--(0,12)(-1.5,8)--(0,6);
\draw[thick](0,6)--(0,12);
\draw[thick](6,0)--(5,6.5)--(0,12);
\draw[thick](0,6)--(5,6.5);
\draw[black,thick](0,12)..controls (-9,11) and (-7,3) .. (-6,0);
\draw[black,thick](0,12)..controls (9,11) and (7,3) .. (6,0);
\node at (0,-2){$P_4(G)=302$};
\draw[fill=black](-6,0)circle(12pt);
\draw[fill=black](6,0)circle(12pt);
\draw[fill=black](0,6)circle(12pt);
\draw[fill=black](0,3)circle(12pt);
\draw[fill=black](0,12)circle(12pt);
\draw[fill=black](-4,5)circle(12pt);
\draw[fill=black](-5,6.5)circle(12pt);
\draw[fill=black](5,6.5)circle(12pt);
\draw[fill=black](-1.5,8)circle(12pt);
\node at (1,3.3){$v$};
\end{tikzpicture}
\caption{Maximal planar graphs on $9$ vertices.}
\label{ddd}
\end{figure}

\item[(iv)] Assume all $4$ vertices are in the same region. 

Suppose there is no star edge, then we have at most \[4\sum_{i=1}^3d(x_i)-24+4\Big(3n-6-(\sum_{i=1}^3d(x_i)-3)\Big)+2(n-6)+6=14n-42=84\] $3$-paths containing the vertex $v$. After removing the vertex $v$ we will not get the unique extremal graph in Figure~\ref{qq}(C), since for each face of the graph in Figure~\ref{qq}(C) has a star edge. Thus, in this case, we have $P_4(G)<222+84=306$. 

Suppose there is a star edge and there is exactly one vertex which is not incident to two of the vertices of the triangle. Then that vertex must be incident to a vertex of the triangle. That vertex of the triangle has degree $8$, therefore after deleting the vertex $v$, we will get a vertex of degree~$7$. Since the graph in Figure~\ref{qq}(C) does not contain a vertex of degree~$7$, the number of $3$-paths not containing $v$ is at most $221$. The number of $3$-paths containing the vertex $v$ is at most \[4\sum_{i=1}^3d(x_i)-24+4(3n-6-(\sum_{i=1}^3d(x_i)-3)-1-1)+5+3+2(n-6)+6=14n-42=84.\]
Thus $P_4(G)< 306 $.

Suppose there is a star edge and there is more than one vertex which is not incident to two of the vertices of the triangle. Then the number of $3$-paths containing the vertex $v$ is at most \[4\sum_{i=1}^3d(x_i)-24+4(3n-6-(\sum_{i=1}^3d(x_i)-3)-1-2)+5+6+2(n-7)+6=14n-42=81.\]Thus $P_4(G)< 306,$ since after deleting the vertex $v$ we get a maximal planar graph on $8$ vertices and $f(8,P_4)=222$. 

Finally, if there is a star edge and all four vertices are incident to two of the vertices of the triangle, then the maximal planar graph is uniquely defined, see Figure~\ref{ext} which is $F_9$. It contains $306$ paths of length three. Therefore $f(9,P_4)=306$, and the unique extremal planar graph on $9$ vertices is $F_9$. 

\end{enumerate}

\begin{figure}[h]
\centering
\begin{tikzpicture}[scale=0.2]
\draw[dashed, red,ultra thick](-6,0)--(0,6)--(6,0)--(-6,0);
\draw[thick](0,3)--(-6,0) (0,3)--(0,6) (0,3)--(6,0);
\draw[thick](-6,0)--(-2.5,5.5)--(0,6); 
\draw[thick](0,6)--(0,12)(-2.5,5.5);
\draw[thick](-6,0)--(-5,8)--(0,12);
\draw[thick](-6,0)--(-4,7)--(0,6);
\draw[thick](-6,0)--(-3.2,6.2)--(0,6);
\draw[thick](0,6)--(-5,8)--(-4,7)--(-3.2,6.2)--(-2.5,5.5);
\draw[black,thick](0,12)..controls (-9,11) and (-7,3) .. (-6,0);
\draw[black,thick](0,12)..controls (9,11) and (7,3) .. (6,0);
\node at (0,-2){$P_4(G)=306$.};
\draw[fill=black](-6,0)circle(10pt);
\draw[fill=black](6,0)circle(10pt);
\draw[fill=black](0,6)circle(10pt);
\draw[fill=black](0,3)circle(10pt);
\draw[fill=black](0,12)circle(10pt);
\draw[fill=black](-4,7)circle(10pt);
\draw[fill=black](-5,8)circle(10pt);
\draw[fill=black](-3.2,6.2)circle(10pt);
\draw[fill=black](-2.5,5.5)circle(10pt);
\node at (1,3.3){$v$};
\end{tikzpicture}
\caption{A maximal planar graph with $9$ vertices, containing maximum number of $3$-paths.}
\label{ext}
\end{figure}
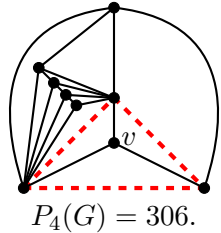

So far we have determined $f(n,P_4)$ for all integers $n$. We also have proven that for all $n$, $n<10$, the planar graph maximizing number of $P_4$'s is unique. Even more we have shown that the unique extremal graph is $F_9$, for $n=9$. 

In the remaining part of this section, we are going to show that for all $n$, $n\geq 9$, the only planar graph maximizing the number of $3$-paths is $F_n$. For this we are going to use a proof by induction on the number of vertices. The base case for $n=9$ is complete. Let us assume that $G$ is an $n$, $n\geq 10$, vertex graph with  $f(n,P_4)$  $3$-paths, then we are going to show that $G=F_n$ under the assumption that the only extremal planar graph with  $(n-1)$ vertices is $F_{n-1}$. From the proof of the upper bound, we know that in order to have $f(n,P_4)$ paths of length three, we have one of two possibilities as outlined in Remark~\ref{equality_triangle}.  

From Lemma~\ref{lemma:d>3} we have the minimum degree of $G$ is $3$, we also have that for any vertex of degree three, say $v$, all other vertices share at least two neighbors with $v$. After removing the vertex $v$, we obtain the unique extremal graph $F_{n-1}$ in this case. Therefore there are only two such faces in $F_{n-1}$, namely the faces with two high degree vertices and a vertex of degree three (the outer face and bottom face from Figure~\ref{qq}(A)). In both settings, after placing $v$ in the proper face and adding all three edges, we obtain the graph $F_n$. Therefore we have the desired result $G=F_n$.
\end{proof}


\section*{Acknowledgements}
The research of the first author is supported by the European Research Council (ERC) grant 648509. The research of the second, fourth and sixth authors is partially supported by the National Research, Development and Innovation Office -- NKFIH, grant K 116769, K 132696 and SNN 117879. The research of the fourth author is partially supported by the Shota Rustaveli National Science Foundation of Georgia SRNSFG, grant number DI-18-118.   The research of the fifth author is supported by the Institute for Basic Science (IBS-R029-C1).


\begin{thebibliography}{1}
\bibitem{ahmad}A. Alameddine. On the number of cycles of length $4$ in a maximal planar graph. \emph{ Journal of Graph Theory}, 4(4) (1980): 417--422.

 \bibitem{alon}N. Alon, Y. Caro. On the number of subgraphs of prescribed  type of planar graphs with a given number of vertices. \emph{Annals of Discrete Mathematics}, 20 (1984): 25--36.
 
 \bibitem{ALS2016} N. Alon, C. Shikhelman. Many $T$ copies in $H$-free graphs. \emph{Journal of Combinatorial Theory, Series B}, 121 (2016): 146--172.
 
 \bibitem{add} N. Alon, C. Shikhelman. Additive approximation of generalized Tur\'an questions. arXiv preprint arXiv:1811.08750 (2018).


\bibitem{BGy2008} B. Bollob\'as, E. Gy\H ori. Pentagons vs. triangles. \emph{Discrete Mathematics}, 308(19) (2008): 4332--4336.

\bibitem{ccox} C. Cox and R.~R. Martin. Counting paths, cycles and blow-ups in planar graphs. arXiv preprint arXiv: 2101.05911v1 (2021).

\bibitem{ccox2} C. Cox and R.~R. Martin. The maximum number of $10$-and $12$-cycles in a planar graph. arXiv preprint arXiv:2106.02966 (2021).
       
\bibitem{E1962} P. Erd\H os. On the number of complete subgraphs contained in certain graphs. \emph{Magyar Tud. Akad. Mat. Kut. Int. K\"ozl.}, 7  (1962): 459--474.

\bibitem{C5C3} B. Ergemlidze, E. Gy\H ori, A. Methuku, N. Salia. A Note on the maximum number of triangles in a $C_5$-free graph. \emph{Journal of Graph Theory}, (2018): 1--4. 
		

\bibitem{C5C3v2} B. Ergemlidze, A. Methuku. Triangles in $C_5$-free graphs and hypergraphs of girth six.  arXiv preprint arXiv:1811.11873 (2018).


 \bibitem{fox}
 J. Fox  F. Wei. On the number of cliques in graphs with a forbidden minor. \emph{Journal of Combinatorial Theory, Series B}, 126 (2017): 175--197.
 
 \bibitem{eppstein}
 D. Eppstein. Connectivity, graph minors, and subgraph multiplicity. \emph{Journal of Graph Theory}, 17.3 (1993): 409--416.
 
 \bibitem{chacha} D. Ghosh, E. Gy\H{o}ri, R.R. Martin, A. Paulos, N. Salia, C. Xiao,  O. Zamora. The maximum number of paths of length four in a planar Graph. \textit{Discrete Mathematics} 344(2021): 112317.
 
 
 \bibitem{gen} E. Gy\H{o}ri, A. Paulos, N. Salia, C. Tompkins, O. Zamora. Generalized Planar Tur\'an Numbers. arXiv preprint arXiv:2002.04579 (2020). 
 
 \bibitem{plpk}
  E. Gy\H{o}ri, N. Salia, C. Tompkins, O. Zamora. The maximum number of $P_\ell$ copies in a $P_k$-free graph. \emph{Discrete Mathematics \& Theoretical Computer Science}, 21(1) (2019).
 
 \bibitem{us} E. Gy\H{o}ri, A. Paulos, N. Salia, C. Tompkins, O. Zamora. The maximum number of pentagons in a planar graph. arXiv preprint arXiv:1909.13532 (2019). 
		
        \bibitem{G2012} A. Grzesik. On the maximum number of five-cycles in a triangle-free graph. \emph{Journal of Combinatorial Theory, Series B}, 102(5) (2012): 1061--1066.

        \bibitem{GK2018} A. Grzesik, B. Kielak. On the maximum number of odd cycles in graphs without smaller odd cycles. arXiv preprint arXiv:1806.09953 (2018).
        

        \bibitem{hakimi}S. Hakimi, E.F. Schmeichel. On the number of cycles of length $k$ in a maximal planar graph. \emph{Journal of Graph Theory}, $3$ (1979): 69--86. 
        
        \bibitem{HHKNR2013} H. Hatami, J. Hladk\'y, D. Kr\' al, S. Norine, A. Razborov. On the number of pentagons in triangle-free graphs.
       \emph{Journal of Combinatorial Theory, Series A}, 120(3) (2013): 722--732.
        
        \bibitem{general} T.~Huynh, D.~Wood. Tree densities in sparse graph classes. arXiv preprint arXiv:2009.12989 (2020).
        
        \bibitem{surfaces} T. Huynh, G. Joret, D. Wood. Subgraph densities in a surface. arXiv preprint arXiv:2003.13777 (2020).
        
        
        \bibitem{k5}
C. Lee and S. Oum. Number of cliques in graphs with a forbidden subdivision. \emph{SIAM Journal on Discrete Mathematics}, 29.4 (2015): 1999--2005.
        
        \bibitem{karatowski}
K. Kuratowski. Sur le probl\'eme des courbes gauches en topologie.\emph{ Fund. Math.} (in French), 15 (1930): 271--283.


\bibitem{k4}
D. Wood. On the maximum number of cliques in a graph. \emph{Graphs and Combinatorics}, 23.3 (2007): 337--352.

\bibitem{kcliques}
D. Wood. Cliques in graphs excluding a complete graph minor. \emph{The Electronic Jounal of Combinatorics}, 23, Issue 3 (2016).

\bibitem{Wormald}
N. Wormald. On the frequency of 3-connected subgraphs of planar graphs. \emph{Bulletin of the Australian Mathematical Society}, 34.2 (1986): 309--317.

\bibitem{zykov}
A. Zykov. On some properties of linear complexes.\emph{ Mat. Sbornik N.~S.}~24(66) (1949): 163--188.
 
\end{thebibliography}
\end{document}